\documentclass{amsart}
\usepackage{amsmath, amsthm, amssymb, amsfonts}
\usepackage{graphicx, tikz, units, url}
\usepackage{color}

\usepackage{mathrsfs} 
\usepackage{bm,bbm} 
\usepackage{colonequals} 
\usepackage{physics} 
\usepackage{extarrows} 
\usepackage{mathtools} 

\usepackage{enumitem}
\usepackage{scalerel} 
\usepackage{stackengine}

\usepackage[colorlinks,citecolor=blue,linkcolor=blue,backref=page]{hyperref}
\usepackage[nameinlink]{cleveref}

\newcommand{\sA}{\mathcal{A}}

\newcommand{\sN}{\mathcal{N}}

\newcommand{\sS}{\mathcal{S}}

\newcommand{\sX}{\mathcal{X}}
\newcommand{\sY}{\mathcal{Y}}

\newcommand{\N}{\mathbb{N}}
\newcommand{\Z}{\mathbb{Z}}

\newcommand{\R}{\mathbb{R}}

\renewcommand{\H}{\mathbb{H}}

\DeclareMathOperator{\Forall}{\forall}
\DeclareMathOperator{\arcsinh}{arcsinh}

\DeclareMathOperator{\length}{\ell} 
\DeclareMathOperator{\area}{Area}

\DeclareMathOperator{\diam}{Diam}

\DeclareMathOperator{\inj}{inj}
\DeclareMathOperator{\codim}{codim}
\DeclareMathOperator{\collar}{\mathcal{C}} 
\DeclareMathOperator{\dist}{dist}

\newcommand{\dif}{\mathop{}\!\mathrm{d}} 

\newcommand{\thick}{X^{\mathrm{thick}}}
\newcommand{\thin}{X^{\mathrm{thin}}}

\newcommand{\supp}{\mathrm{supp}}

\newcommand{\pa}{\partial}
\newcommand{\df}{:=}

\newcommand{\eg}{\textit{e.g.,\@ }}

\newcommand{\ie}{\textit{i.e.\@ }}

\theoremstyle{plain}
\newtheorem{theorem}{Theorem}[section] 
\newtheorem{lemma}[theorem]{Lemma}

\newtheorem*{construction*}{Construction}
\newtheorem*{theorem*}{\bf Theorem}
\newtheorem*{corollary*}{Corollary}
\newtheorem*{assumption*}{Assumption $(\star)$}
\newtheorem*{observation*}{Observation}
\newtheorem*{claim*}{Claim}

\theoremstyle{remark}

\newtheorem*{remark*}{Remark}
\newtheorem*{question*}{Question}

\theoremstyle{definition}
\newtheorem{definition}[theorem]{Definition} 
\newtheorem*{def*}{Definition}

\begin{document}

\title[Multiplicities of eigenvalues]{Short geodesics and multiplicities of eigenvalues of hyperbolic surfaces}

\author{Xiang He}
\address[X. ~H. ]{Tsinghua University, Beijing, China}
\email{x-he@mail.tsinghua.edu.cn}

\author{Yunhui Wu}
\address[Y. ~W. ]{Tsinghua University \& BIMSA, Beijing, China}
\email{yunhui\_wu@mail.tsinghua.edu.cn} 

\author{Haohao Zhang}
\address[H. ~Z. ]{Tsinghua University, Beijing, China}
\email{zhh21@mails.tsinghua.edu.cn}

\begin{abstract}
In this paper, we obtain upper bounds on the multiplicity of Laplacian eigenvalues for closed hyperbolic surfaces in terms of the number of short closed geodesics and the genus $g$. For example, we show that if the number of short closed geodesics is sublinear in $g$, then the multiplicity of the first eigenvalue is also sublinear in $g$. This makes new progress on a conjecture by Colin de Verdi\`ere in the mid 1980s.
\end{abstract}

\maketitle

\section{Introduction}

Let $X_g$ be a closed hyperbolic surface of genus $g \ge 2$, and let $\Delta$ denote the Laplacian on $X_g$. By classical results in spectral theory, the spectrum of $\Delta$ consists of a sequence of discrete nonnegative eigenvalues counted with multiplicities:
\[
0=\lambda_0(X_g)<\lambda_1(X_g)\le\lambda_2(X_g)\le\cdots\to+\infty.
\]
For any $\lambda>0$, denote by $m(\lambda)$ the multiplicity of $\lambda$ in the spectrum of $\Delta$. In this paper, we are interested in studying the upper bounds of their multiplicities, especially in $m(\lambda_1)$.  

For any constant $\epsilon>0$ that may depend on $g$, define 
\[
\sN_\epsilon(X_g) := \{ \gamma;\ \text{$\gamma\subset X_g$ is a simple closed geodesic of length $<2\epsilon$} \}, 
\]
and set $N_\epsilon(X_g) = \#{\sN_\epsilon(X_g)}$. Our first result is as follows. 
\begin{theorem}\label{thm:mt-1}
Let $X_g$ be a closed hyperbolic surface of genus $g$. Then there exists a uniform constant $K\geq 1$ such that for any constant $\epsilon\in (0,1)$, the multiplicity of $\lambda_1(X_g)$ satisfies
\[m(\lambda_1(X_g))\le \frac{K}{\epsilon^2} \cdot \frac{g}{\log \log \left( \frac{10g}{N_\epsilon(X_g)+1}\right)}.\]
\end{theorem}
\begin{remark*}
\begin{enumerate}
    \item If $N_{\epsilon_0}(X_g)\leq f(g)=o(g)$ for some $\epsilon_0>0$ and positive function $f:\mathbb{Z}^{\geq 2}\to \mathbb{Z}$ with sublinear growth, then Theorem \ref{thm:mt-1} tells that $$m(\lambda_1(X_g))\ll \frac{g}{\log \log \left(\frac{g}{f(g)} \right)}.$$
    This upper bound is new and also sublinear in $g$. In particular, if $f(g)\ll g^\alpha$ for any fixed $\alpha \in (0,1)$, then we have
    $$m(\lambda_1(X_g))\ll \frac{g}{\log \log g}.$$
See Figure \ref{fig:nsep} for an example. 
    \item If $N_{\epsilon_0}(X_g)=0$ for some $\epsilon_0>0$, which is equivalent to saying that $X_g$ has injectivity radius $\geq \epsilon_0$, as above Theorem \ref{thm:mt-1} tells that  $$m(\lambda_1(X_g))\ll \frac{g}{\log \log g}.$$ This was due to Letrouit--Machado \cite{LM2024} for this case. We remark here that Jiang--Tidor--Yao--Zhang--Zhao \cite{JTYZZ2021} proved the same sublinear upper bound for the second largest eigenvalue of connected $g$-vertex graphs with bounded maximum degree. 
    \item By the classical Collar Lemma, it is known that $N_\epsilon(X_g)\leq 3g-3$ whenever $\epsilon$ is small. Thus, if we choose $\epsilon=0.5$, Theorem \ref{thm:mt-1} then tells that for any closed hyperbolic surface $X_g$ of genus $g$,  $$m(\lambda_1(X_g))\ll g.$$
    The linear upper bound was first obtained by Besson \cite{Besson1980}.
\end{enumerate}
\end{remark*}
\begin{figure}
    \centering
    \includegraphics[scale=0.26]{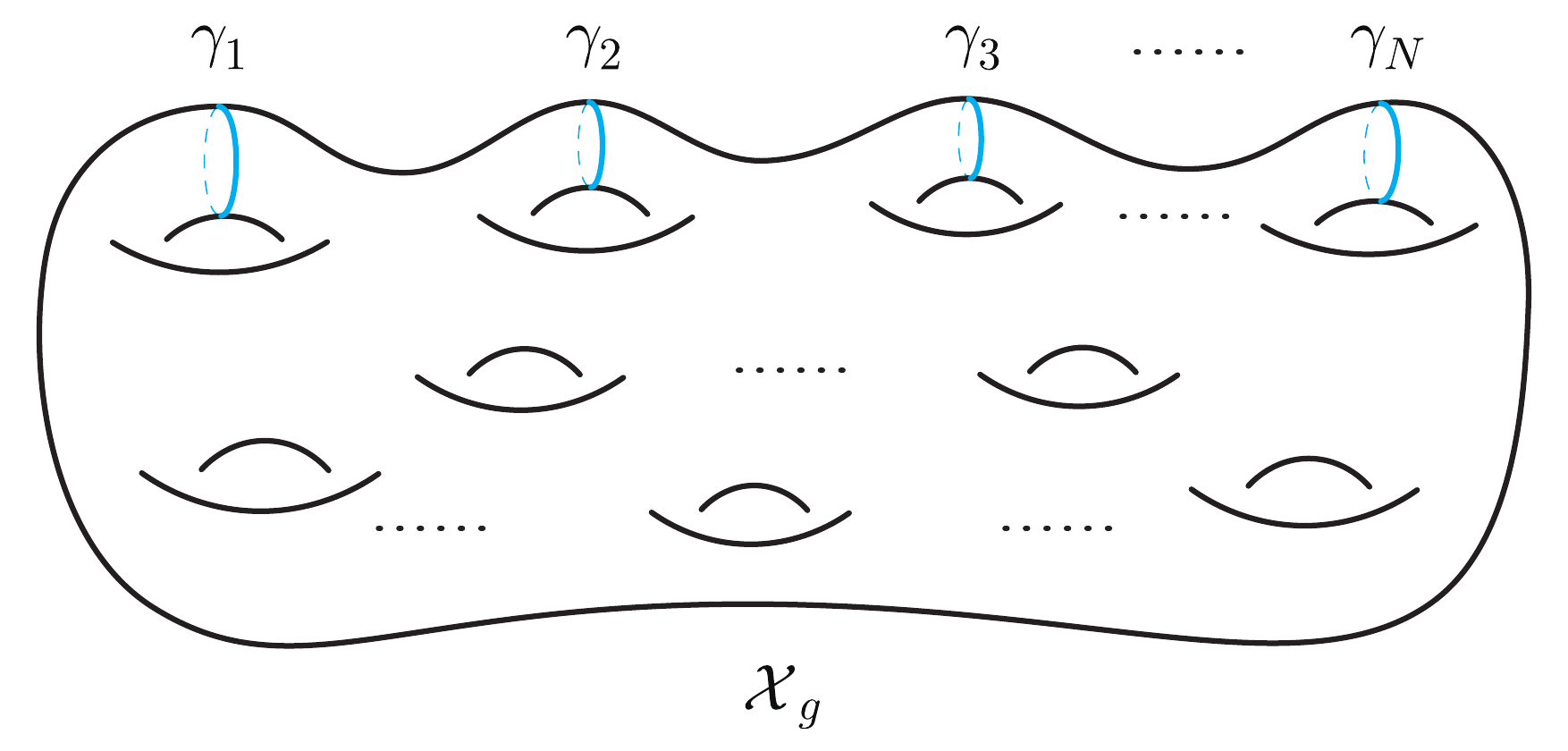}
    \caption{A closed hyperbolic surface $\sX_g$ of genus $g$, where $\gamma_i\in\sN_{\epsilon_0}(\sX_g)$ and $N_{\epsilon_0}(\sX_g)=o(g)$ for some fixed $\epsilon_0>0$. So $m(\lambda_1(\sX_g))=o(g)$.}
    \label{fig:nsep}
\end{figure}

A key novel ingredient in proving Theorem \ref{thm:mt-1} is to find a linear subspace $S$ of the eigenspace $E(\lambda_1(X_g))$ with respect to $\lambda_1(X_g)$ such that the dimension $\dim(S)$ of the subspace $S$ is uniformly comparable to $\dim(E(\lambda_1(X_g)))=m_1(\lambda_1(X_g))$, and the $L^2$-norm of each eigenfunction in $S$ on $X_g$ is uniformly comparable to its $L^2$-norm on the $\epsilon$-thick part of $X_g$. And then restricted to the thick part of $X_g$, we follow a strategy similar to that in the recent works \cite{JTYZZ2021, LM2024} to complete the proof of Theorem \ref{thm:mt-1}.  \\

Our second result concerns the multiplicities of small eigenvalues. For any $\delta\in(0,1/2)$, set 
\[ \epsilon(\delta) = \arcsinh\frac{1}{\sinh(1/\delta+2)}, \]
and for any $\epsilon\in(0,\epsilon(\delta))$, set 
\[I_{\epsilon}(X_g)=\textit{the number of components of }X_g\setminus \sN_{\epsilon}(X_g).\]
Our second theorem is as follows. 
\begin{theorem}\label{thm:mt-2}
 Let $X_g$ be a closed hyperbolic surface of genus $g$, and let $\lambda>0$ be an eigenvalue of $X_g$. If $\lambda\leq1/4-\delta^2$, then for any $\epsilon\in(0,\epsilon(\delta))$, there exists a constant $C(\epsilon)>0$ only depending on $\epsilon$ such that
\[ m(\lambda) \leq C(\epsilon) \sqrt{\lambda}\cdot g + 24I_{\epsilon}(X_g) . \] 
\end{theorem}

\begin{remark*}
    \begin{enumerate}
        \item If $\lambda_1(X_g)=o(1)$ and $I_{\epsilon_0}(X_g) = o(g)$ for some fixed $\epsilon_0$ as $g\to\infty$, then Theorem \ref{thm:mt-2} tells that $m(\lambda_1(X_g))=o(g)$. Note that $N_{\epsilon_0}(X_g)$ may have growth rate $g$, in which case Theorem \ref{thm:mt-1} only gives a linear upper bound for $m(\lambda_1(X_g))$. See Figure \ref{fig:sep} for an example. 
        \item Suppose $\epsilon\ll 1/(\log g)^{\kappa}$ for some $\kappa>0$ as $g\to\infty$. If $\sN_{\epsilon}(X_g)$ separates $X_g$, and $I_{\epsilon_0}(X_g)\ll g/(\log g)^{\kappa}$ for some fixed $\epsilon_0$, then the upper bound given by Theorem \ref{thm:mt-2} is slightly stronger than that given by Theorem \ref{thm:mt-1}. 
    \end{enumerate}
\end{remark*}
\begin{figure}
    \centering
    \includegraphics[scale=0.28]{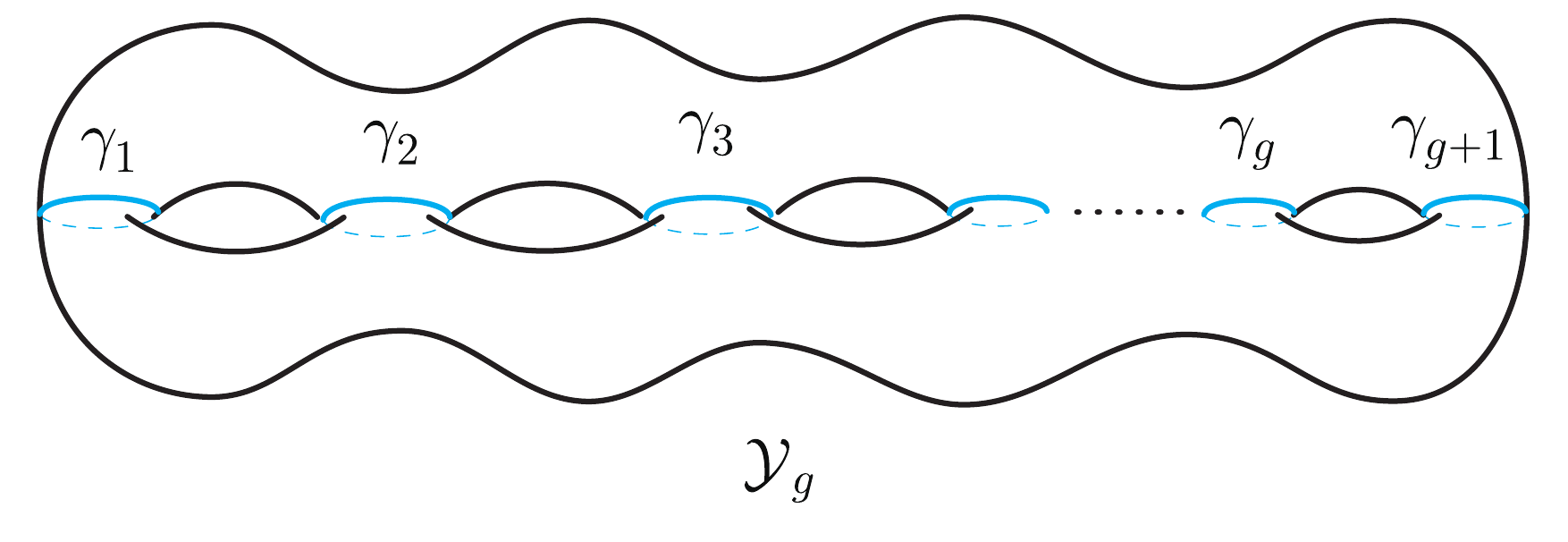}
    \caption{A closed hyperbolic surface $\sY_g$ of genus $g$ separated by $g+1$ geodesics into two pieces, where $\gamma_i\in\sN_{\epsilon}(\sY_g)$ and for each $i$, $\ell(\gamma_i)\to0$ as $g\to\infty$. So $m(\lambda_1(\sY_g))=o(g)$.}
    \label{fig:sep}
\end{figure}

We remark here that for a family of closed hyperbolic surfaces $\{X_g^t\}_{t\in (0,1]}$ of genus $g$ that has $(3g-3)$ pairwise disjoint simple closed geodesics simultaneously pinching to zero as $t\to 0$, it is known from \eg \cite{SWY80} or \cite[Proof of Theorem 21]{HW25} that $\lambda_{2g-3}(X_g^t)\to 0$ uniformly as $t\to 0$. However, both Theorem \ref{thm:mt-1} and Theorem \ref{thm:mt-2} cannot distinguish the first $(2g-3)$ nonzero eigenvalues of this particular family $\{X_g^t\}_{t\in (0,1]}$. To our knowledge, it is not even known whether the multiplicity $m(\lambda_1(X_g^t))$ is sublinear in $g$ when $t$ is close to $0$.

\subsection*{Related works}

The first upper bound on the multiplicity of Laplacian eigenvalues $\{\lambda_i\}$ on closed Riemannian surfaces was established by Cheng \cite{Cheng1976}, who showed that for any closed Riemannian surface of genus $g$, the multiplicity of the $i$-th eigenvalue $\lambda_i$ satisfies
\[ m(\lambda_i)\le \frac 12 (2g+i+1)(2g+i+2),\qquad \Forall i\ge 1. \] 
This was later improved by Besson \cite{Besson1980}, establishing the first linear upper bound: 
\[ m(\lambda_i)\le 4g+2i+1. \] 
Many refinements of Besson's bound have since been obtained. Nadirashvili \cite{Nadirashvili1988} improved the bound for $g\ge2$ to $4g+2i-1$; S\'evennec \cite{Sevennec2002} improved it further for $g\ge2$, $i=1$ to $2g+3$; and Bourque--Petri \cite{BP2023} improved the bound for $i=1$ to $2g-1$ for hyperbolic surfaces of sufficiently large genus. 

Several sublinear upper bounds on the multiplicity of the first nonzero Laplacian eigenvalue on hyperbolic surfaces have been established under certain assumptions. As introduced above, if the surface has injectivity radius $\geq\epsilon>0$, Letrouit--Machado \cite{LM2024} proved that 
$$ m(\lambda_1) \ll_{\epsilon} \frac{g}{\log\log g}. $$
Gilmore--Le Masson--Sahlsten--Thomas \cite{GLeMST21} proved that if the surface has injectivity radius $\geq\frac{1}{(\log g)^\alpha}$, and at each point there is at most one geodesic loop of length less than $c\log g$ based there, then 
\begin{align*}
    m(\lambda) &\ll_{c,\lambda} \frac{g}{(\log g)^{1-\alpha}} \text{\quad for\ } \lambda\in\Big(\frac14, +\infty\Big), \\
    \text{and\quad} m(\lambda) &\ll_{c} g^{1-2c\sqrt{\epsilon}}(\log g)^{2\alpha} \text{\quad for\ } \lambda\in\Big(0,\frac14-\epsilon\Big). 
\end{align*}
Monk \cite{Monk22} proved that if the surface has injectivity radius $\ge g^{-\frac 1{24}}(\log g)^{\frac 9{16}}$ and the area of the points where the injectivity radius $<\frac 16\log g$ is $O(g^{\frac 23})$, then
\[
m(\lambda)\ll \sqrt{1+\lambda}\frac{g}{\sqrt{\log g}}.
\]
We emphasize here that our upper bound in Theorem \ref{thm:mt-1} and Theorem \ref{thm:mt-2} do not require any lower bound of the injectivity radius, instead we need the upper bounds of the number of short closed geodesics or components of the thick part of the surface. 
Bourque--Petri \cite{BP2023} proved that, for every $p_1,p_2\in(j_0,\pi]$, where $j_0$ is the first positive zero of the Bessel function $J_0$,  
$$ m(\lambda_1) \ll_{p_1,p_2} \frac{g}{(\log g)^3}, \text{\quad when\ $\lambda_1\in \bigg[ \frac14+\Big(\frac{p_1}{\log g}\Big)^2, \frac14+\Big(\frac{p_2}{\log g}\Big)^2 \bigg]$}. $$

\subsubsection*{Colin de Verdi{\`e}re's conjecture}
In \cite{CdV86}, Colin de Verdi{\`e}re  \emph{conjectured} that the multiplicity of the first nonzero Laplacian eigenvalue on a closed orientable surface $M_g$ of genus $g$ is bounded from above by
\[
\mathrm{chr}(M_g)-1,
\]
where the \emph{chromatic number} $\mathrm{chr}(M_g)$ is the supremum of the natural numbers $n$ for which the complete graph on $n$ vertices can be embedded in $M_g$.  It is known from \cite{RY68} of Ringel--Youngs that
\[
\mathrm{chr}(M_g) = \Big\lfloor\frac 12\big(7+\sqrt{48g+1}\big) \Big\rfloor.
\]
In support of this, Colbois and Colin de Verdi{\`e}re constructed, in \cite{CCdV88}, for every $g\ge 3$, a hyperbolic surface $X_g$ such that $$m(\lambda_1(X_g))=\Big\lfloor\frac{1+\sqrt{8g+1}}2\Big\rfloor,$$
demonstrating that the maximal multiplicity grows at least on the order of $\sqrt g$. However, counterexamples to the conjecture are found for $g = 10$ and $g = 17$ by Bourque--Gruda-Mediavilla--Petri--Pineault in \cite{BGMPP23}.

Nevertheless, for closed hyperbolic surfaces of high genus, it is quite interesting to \emph{determine the right growth rate of the maximal multiplicity of the first nonzero eigenvalue of the Laplacian. Is it $\sqrt{g}$ in the above conjecture of Colin de Verdi\`ere, or $\frac{g}{\log \log g}$ as shown above, or something else?}

\subsection*{Notations.} For any two positive functions $f(g)$ and $h(g)$, we say $f=o(h)$ if $\lim\limits_{g\to \infty}\frac{f(g)}{h(g)}=0$; we say $f\ll h$ if there exists a uniform constant $C>0$ such that $f\leq Ch$. We use $\lfloor \cdot \rfloor$ to denote the greatest integer function.

\subsection*{Plan of the paper}
In Section \ref{section2} we introduce the standard thick-thin decomposition of closed hyperbolic surfaces. In Section \ref{section3} we use the number of short closed geodesics to construct a subspace of $\lambda$-eigenspace on which each eigenvalue has $L^2$-norm concentrated in the thick part, which will be applied in the proofs of Theorem \ref{thm:mt-1} and Theorem \ref{thm:mt-2}. In Section \ref{section4} we prove Theorem \ref{thm:mt-1}. And we prove Theorem \ref{thm:mt-2} in Section \ref{section5}.

\subsection*{Acknowledgements} We would like to thank all the participants in our seminar at Tsinghua University on Teichm\"uller theory for helpful discussions on this project. The first named author is supported by the China Postdoctoral Science Foundation No. 2024M761591. The second named author is partially supported by NSFC grants No. 12171263, 12361141813, and 12425107.

\tableofcontents

\section{Thick-thin decomposition}\label{section2}
Let $X_g$ be a closed hyperbolic surface of genus $g$, and let $\dist(\cdot,\cdot)$ denote the distance function on $X_g$. For any simple closed geodesic $\gamma\subset X_g$, the collar around $\gamma$ of width $w$ is defined by 
\[ \collar(\gamma,w) \df \{x\in X_g| \ \dist(x,\gamma)\leq w \}. \] 
Let $\ell(\gamma)$ be the length of $\gamma$. The Collar Lemma asserts that the collar of width 
\begin{equation}\label{eqn:collarwidth}
    w(\gamma) \df \arcsinh\frac{1}{\sinh\frac{\ell(\gamma)}{2}} 
\end{equation}
is always embedded in $X_g$. We call it the standard collar of $\gamma$. Moreover, $\collar(\gamma,w(\gamma))$ is isometric to $[-w(\gamma),w(\gamma)]\times \R/\Z$ endowed with the hyperbolic metric 
\begin{equation*}
    \dif\rho^2 + \ell(\gamma)^2\cosh^2\!\rho\dif t^2. 
\end{equation*}
The pair of coordinates $(\rho,t)$ is called the Fermi coordinates on $\collar(\gamma,w(\gamma))$. For any $x\in X_g$, let $\inj(x)$ denote the injectivity radius at $x$. We recall that 
\begin{lemma}[Collar Lemma, {\cite[Theorem 4.1.6]{Buser1992}}]\label{lem:collar}
    Let $\{\gamma_1,\cdots,\gamma_N\}$ be the set of all simple closed geodesics of length $\leq2\arcsinh1$ in $X_g$, and let $\collar_i$ be the standard collar of $\gamma_i$. Then the followings hold.
    \begin{enumerate}
        \item $N\leq 3g-3$. 
        \item The collars $\collar_i$ are pairwise disjoint. 
        \item Suppose $x\in \collar_i$ and $\dist(x,\pa\collar_i) = d$, then 
        \begin{equation}\label{eqn:inj}
                \sinh\inj(x) = \cosh\frac{\ell(\gamma_i)}{2}\cosh d - \sinh d. 
            \end{equation}
    \end{enumerate}
\end{lemma}

Without loss of generality, we now fix $\epsilon$ so that 
\begin{equation}\label{eqn:epsilon}
    0< \epsilon < \arcsinh\frac{1}{\sinh2} \approx 0.272. 
\end{equation}
Recall that 
\[ \sN_\epsilon(X_g) = \{ \gamma;\ \text{$\gamma\subset X_g$ is a simple closed geodesic of length $<2\epsilon$} \}. \] 
Since $\epsilon < \arcsinh1$, by Part (2) of Lemma \ref{lem:collar}, $\sN_\epsilon(X_g)$ consists of pairwise disjoint simple closed geodesics. Moreover, by \eqref{eqn:collarwidth}, for any $\gamma\in \sN_\epsilon(X_g)$ we have 
\[ w(\gamma)>1. \] 
In the proof of Theorem \ref{thm:mt-1}, we will use a thick-thin decomposition of $X_g$ which is defined as follows. 

\begin{definition}\label{def:thick-thin}
    The thin part of $X_g$ is defined as
    \begin{equation*}
        \thin_g \df \bigcup_{\gamma\in \sN_\epsilon(X_g)}\collar(\gamma,w(\gamma)-1). 
    \end{equation*}
    The thick part of $X_g$ is defined as the closure of its complement, which is given by 
    \begin{equation*}
        \thick_g \df \overline{X\setminus\thin_g} = \sA_1\sqcup\sA_2\sqcup\cdots\sqcup\sA_I. 
    \end{equation*}
    Here $I\geq 1$, and $\sA_k$, $1\leq k\leq I$, are all components of $\thick_g$. 
\end{definition}
\noindent Then we have the following. 
\begin{itemize}
    \item By \eqref{eqn:inj} and \eqref{eqn:epsilon}, for any $x\in \thick_g$ we have 
    \begin{equation}\label{injlb}
    \inj(x)\geq \min\big\{\epsilon,\, \arcsinh(\cosh1-\sinh1)\big\}=\epsilon. 
    \end{equation}
    \item $\Forall \gamma\in\sN_\varepsilon(X_g)$, the length of a component of $\pa\collar(\gamma,w(\gamma)-1)$ is equal to 
    \begin{equation}\label{eqn:lw-1}
        \begin{aligned}
            &\ell(\gamma)\cosh(w(\gamma)-1)
            \leq \frac2e \ell(\gamma)\cosh w(\gamma)\\
            &= \frac2e \sqrt{\ell(\gamma)^2 + \frac{\ell(\gamma)^2}{\sinh^2(\ell(\gamma)/2)}}\leq \frac2e \sqrt{1+4} < 2. 
        \end{aligned}
    \end{equation}
\end{itemize}

\section{Mass distribution of eigenfunctions}\label{section3}

One key ingredient in the proof of Theorem \ref{thm:mt-1} is that, if $m(\lambda)$ is large (compared to the number of short closed geodesics $N_\epsilon$), then we can find a subspace of the eigenspace $E(\lambda)$ such that the $L^2$-norm of any eigenfunction in this subspace is concentrated on the thick part of $X_g$. 

Let $f$ be an eigenfunction of $\Delta$ with eigenvalue $\lambda$, and let $\collar(\gamma,w)$ be an embedded collar in $X_g$. In the Fermi coordinates $(\rho,t)$, the Laplacian is given by 
\begin{equation*}
    -\Delta = \frac{\pa^2}{\pa\rho^2} + \tanh\rho\frac{\pa}{\pa\rho} + \frac{1}{\ell(\gamma)^2\cosh^2\!\rho}\frac{\pa^2}{\pa t^2}. 
\end{equation*}
Suppose $f$ has the following Fourier expansion on $\collar(\gamma,w)$, 
\begin{equation}\label{eqn:fourier}
    f(\rho,t) = \alpha_0(\rho) + \sum_{j=1}^\infty \big(\alpha_j(\rho)\cos(2\pi jt)+\beta_j(\rho)\sin(2\pi jt)\big). 
\end{equation}
Then $\sqrt{\cosh\rho}\cdot \alpha_j$ and $\sqrt{\cosh\rho}\cdot \beta_j$ satisfy the following differential equation (one may refer to \eg \cite[Section 4]{Gamburd2002}, \cite[Section 3]{Mondal2015} and \cite[Section 5]{WZ2025}  for more details): 
\begin{equation}\label{eqn:u}
    \frac{\dif^2u}{\dif\rho^2} = \bigg( \frac14-\lambda + \Big(\frac14 + \frac{4\pi^2j^2}{\ell(\gamma)^2}\Big)\frac{1}{\cosh^2\rho} \bigg)u. 
\end{equation}
Let $\varphi_j$, $\psi_j$ be two linearly independent solutions of \eqref{eqn:u} satisfying 
\[ \varphi_j(0) = \frac{\dif\psi_j}{\dif\rho}(0) = 0, \textrm{\quad and\quad} \frac{\dif\varphi_j}{\dif\rho}(0) = \psi_j(0) = 1. \] 
Recall that $w(\gamma)$ is the width of the standard collar of $\gamma$. We have the following elementary lemma: 
\begin{lemma}\label{lem:mass-1}
    Suppose that $\ell(\gamma)\leq 1$. If $j\geq\lfloor\sqrt{\lambda}\rfloor+1$, then 
    \begin{equation*}
        \max\bigg\{ \frac{\int_0^{w(\gamma)-1}\varphi_j^2 \dif\rho}{\int_0^{w(\gamma)}\varphi_j^2 \dif\rho},\ \frac{\int_0^{w(\gamma)-1}\psi_j^2 \dif\rho}{\int_0^{w(\gamma)}\psi_j^2 \dif\rho} \bigg\} \leq \frac{4}{e^{2}}. 
    \end{equation*}
\end{lemma}
\begin{proof}
    Suppose that $\rho\leq w(\gamma)$. By \eqref{eqn:collarwidth} and our assumption $\ell(\gamma)\leq1$, we have 
    \[ \ell(\gamma)^2\cosh^2\!\rho \leq \ell(\gamma)^2 + \frac{\ell(\gamma)^2}{\sinh^2\!\frac{\ell(\gamma)}{2}} \leq 5. \]
    Therefore if $j\geq\lfloor\sqrt{\lambda}\rfloor+1$, then 
    \begin{equation*}
        \frac14-\lambda + \Big(\frac14 + \frac{4\pi^2j^2}{\ell(\gamma)^2}\Big)\frac{1}{\cosh^2\rho} \geq \frac{4\pi^2j^2}{5}-\lambda \geq 1. 
    \end{equation*}
    It follows that $\varphi_j$ and $\psi_j$ satisfy
    \begin{equation*}
        \frac{\dif^2\varphi_j}{\dif\rho^2} \geq \varphi_j \text{\quad and\quad} \frac{\dif^2\psi_j}{\dif\rho^2} \geq \psi_j. 
    \end{equation*}
    By plugging $u_1=\varphi_j \textit{ or } \psi_j$ and $u_2=\cosh\rho$ into \cite[Lemma 5.2]{WZ2025}, then we have 
    \begin{equation}\label{i-c-tt}
         \begin{aligned}
       \max\bigg\{ \frac{\int_0^{w(\gamma)-1}\varphi_j^2 \dif\rho}{\int_0^{w(\gamma)}\varphi_j^2 \dif\rho},\ \frac{\int_0^{w(\gamma)-1}\psi_j^2 \dif\rho}{\int_0^{w(\gamma)}\psi_j^2 \dif\rho} \bigg\}
       &\leq \frac{\int_0^{w(\gamma)-1}\cosh^2\!\rho \dif\rho}{\int_0^{w(\gamma)}\cosh^2\!\rho \dif\rho}\\ 
            &  =  \frac{2w(\gamma)-2+\sinh(2w(\gamma)-2)}{2w(\gamma)+\sinh(2w(\gamma))}\\
            &\leq \frac{3\exp(2w(\gamma)-2)}{\exp(2w(\gamma))-\exp(-2w(\gamma))}\\
            &\leq \frac{4}{e^{2}}, 
         \end{aligned}
    \end{equation}
     where we use the fact that $\exp(2w(\gamma))-\exp(-2w(\gamma))\geq 3\exp(2w(\gamma))/4$. 
    The proof is complete. 
\end{proof}

Next, for any $j\geq0$ and any simple closed geodesic $\gamma\subset X_g$, we define 
\begin{equation*}
    \begin{aligned}
        \varphi_{j,\gamma}^1(x) = \begin{cases}
            \varphi_j(\rho)\cos(2\pi jt), &x=(\rho,t)\in\collar(\gamma,w(\gamma));\\
            0, &x\notin\collar(\gamma,w(\gamma)). 
        \end{cases}
    \end{aligned}
\end{equation*}  
For any $j\ge 1$, we define
\begin{equation*}
    \begin{aligned}
        \varphi_{j,\gamma}^2(x) = \begin{cases}
            \varphi_j(\rho)\sin(2\pi jt), &x=(\rho,t)\in\collar(\gamma,w(\gamma));\\
            0, &x\notin\collar(\gamma,w(\gamma)). 
        \end{cases}
        \end{aligned}
\end{equation*}  
And let $\psi_{j,\gamma}^1, \psi_{j,\gamma}^2\in L^2(X_g)$ be defined similarly. For all $j\geq 0$, we now define a subspace $S(j,\gamma)$ of $L^2(X_g)$ as
\begin{equation*}
 S(j,\gamma):=\mathrm{span}\big\{ \varphi_{0,\gamma}^1,\varphi_{1,\gamma}^1,\varphi_{1,\gamma}^2,\cdots,\varphi_{j,\gamma}^1,\varphi_{j,\gamma}^2, \psi_{0,\gamma}^1,\psi_{1,\gamma}^1,\psi_{1,\gamma}^2,\cdots,\psi_{j,\gamma}^1,\psi_{j,\gamma}^2 \big\}.
\end{equation*}
Then we have 
\begin{lemma}\label{lem:mass-2}
    Let $f$ be an eigenfunction with eigenvalue $\lambda$, and let $\gamma$ be a simple closed geodesic of length $\leq1$. If $f\in S(\lfloor\sqrt{\lambda}\rfloor,\gamma)^{\bot}$, the orthogonal complement of $S(\lfloor\sqrt{\lambda}\rfloor,\gamma)$, then 
    \begin{equation*}
        \int_{\collar(\gamma,w(\gamma)-1)}f^2 \dif A \leq 2\int_{\collar(\gamma,w(\gamma))\setminus\collar(\gamma,w(\gamma)-1)}f^2 \dif A. 
    \end{equation*}
\end{lemma}
\begin{proof}
    We use the same notation as in \eqref{eqn:fourier}. Assume that 
    \begin{equation*}
        \begin{aligned}
            \sqrt{\cosh\rho}\cdot\alpha_j(\rho) &= a_{j1}\cdot\varphi_j(\rho) + a_{j2}\cdot\psi_j(\rho),\\
            \sqrt{\cosh\rho}\cdot\beta_j(\rho) &= b_{j1}\cdot\varphi_j(\rho) + b_{j2}\cdot\psi_j(\rho).
        \end{aligned}
    \end{equation*}
    Then the $L^2$-norm of $f$ on the collar $\collar(\gamma,w)$ is 
    \begin{equation}\label{eqn:mass-2-1}
        \int_{\collar(\gamma,w)}f^2 \dif A = 2\ell^2(\gamma)\cdot \int_{0}^{w} \Big(a_{01}^2\varphi_0^2+a_{02}^2\psi_0^2 + \frac12\sum_{j=1}^\infty(c_{j1}^2\varphi_j^2+c_{j2}^2\psi_j^2)\Big) \dif\rho, 
    \end{equation}
    where $c_{ji}^2=a_{ji}^2+b_{ji}^2$. If $f\in S(\lfloor\sqrt{\lambda}\rfloor,\gamma)^{\bot}$, then 
    \begin{equation}\label{eqn:mass-2-2}
        a_{01}=a_{02}=0,\text{\quad and\quad} c_{j1}=c_{j2}=0,\ 1\leq j\leq \lfloor\sqrt{\lambda}\rfloor. 
    \end{equation}
    It then follows from \eqref{eqn:mass-2-1}, \eqref{eqn:mass-2-2} and Lemma \ref{lem:mass-1} that 
    \begin{equation*}
        \begin{aligned}
            \int_{\collar(\gamma,w(\gamma)-1)}f^2 \dif A 
            &\leq \frac{4/e^2}{1-4/e^2} \int_{\collar(\gamma,w(\gamma))\setminus\collar(\gamma,w(\gamma)-1)}f^2 \dif A\\
            &  < 2\int_{\collar(\gamma,w(\gamma))\setminus\collar(\gamma,w(\gamma)-1)}f^2 \dif A. 
        \end{aligned}
    \end{equation*}
    The proof is complete. 
\end{proof}

Recall that for $0<\epsilon<1$,  
\[ \sN_\epsilon(X_g) = \{ \gamma;\ \text{$\gamma\subset X_g$ is a simple closed geodesic of length $<2\epsilon$} \}. \] 
We have the following lemma. 
\begin{lemma}\label{lem:mass}
    Let $X_g$ be a closed hyperbolic surface, and let $E(\lambda)\subset L^2(X_g)$ be an eigenspace of $\Delta$ with respect to the eigenvalue $\lambda$. Then there exists a subspace $S$ of $E(\lambda)$ satisfying 
    \[ \codim S \leq 4(\lfloor\sqrt{\lambda}\rfloor+1)\cdot N_\epsilon(X_g) \] 
    such that $\Forall f\in S$ (if $S\neq 0$), 
    \[ \int_{X_g} f^2 \dif A \leq 3\int_{\thick_g} f^2 \dif A, \] 
    where $\thick_g$ is the thick part of $X_g$ corresponding to $\sN_{\epsilon}(X_g)$. 
\end{lemma}
\begin{proof}
    Consider the orthogonal projection 
    \[ \Pi: E(\lambda) \to \bigoplus_{\gamma\in \sN_\epsilon(X_g)}S(\lfloor\sqrt{\lambda}\rfloor,\gamma). \] 
    Then 
    \[ \codim(\ker\Pi) \leq \dim \Big(\bigoplus_{\gamma\in \sN_\epsilon(X_g)}S(\lfloor\sqrt{\lambda}\rfloor,\gamma)\Big) \leq 4(\lfloor\sqrt{\lambda}\rfloor+1)\cdot N_\epsilon(X_g), \]
    and by Lemma \ref{lem:mass-2}, for any $f\in \ker\Pi$ (if $\ker\Pi\neq 0$) we have 
    \begin{equation*}
        \begin{aligned}
        \int_{X_g} f^2 \dif A 
            &= \int_{\bigcup_{\gamma\in \sN_\epsilon(X_g)}\collar(\gamma,w(\gamma)-1)} f^2 \dif A + \int_{\thick_g} f^2 \dif A\\
            &\leq 2 \int_{\bigcup_{\gamma\in \sN_\epsilon(X_g)}\collar(\gamma,w(\gamma))\setminus\collar(\gamma,w(\gamma)-1)} f^2 \dif A + \int_{\thick_g} f^2 \dif A\\
            &\leq 3 \int_{\thick_g} f^2 \dif A. 
        \end{aligned}
    \end{equation*}
    Thus, the subspace $\ker\Pi\subset E(\lambda)$ satisfies the desired properties. 
\end{proof}

\section{Proof of Theorem \ref{thm:mt-1}}\label{section4}
In this section, we prove Theorem \ref{thm:mt-1}. 

Let $X_g$ be a closed hyperbolic surface of genus $g$. Recall that 
\[ \epsilon \in \Big(0, \arcsinh\frac{1}{\sinh2}\Big) \]
is a fixed constant, and $N_\epsilon(X_g)=\#\sN_\epsilon(X_g)$. By the general upper bound of the first nonzero eigenvalue (\eg \cite[Corollary 2.3]{Cheng1975} or \cite[Proposition 9.1]{WZ2025}), one has that as $g\to \infty$, 
\[
\lambda_1(X_g)\le \frac {1}4+o(1).
\]
Now we always assume $$\lambda_1(X_{g}) < 1$$ throughout this section by taking $g$ sufficiently large. 

\subsection{Rescaling the metric}
We first rescale the hyperbolic metric on $X_g$ by $\epsilon^{-2}$ to obtain a new Riemannian surface $X_{g,\epsilon}$. Then the 2-dimensional simply connected space form of curvature $-\epsilon^2$, denoted by $\mathbb H_\epsilon$, is the universal cover of $X_{g,\epsilon}$.
Let $\dist_\epsilon(\cdot,\cdot)$ denote the distance function on $X_{g,\epsilon}$, and for any $x\in X_{g,\epsilon}$, let $B_\epsilon(x,r)$ denote the geodesic ball of radius $r>0$ centered at $x$. 

Let $\Delta_\epsilon$ be the Laplacian on $X_{g,\epsilon}$. Then $\lambda$ is an eigenvalue of the Laplacian $\Delta$ on $X_g$ if and only if $\epsilon^2\lambda$ is an eigenvalue of $\Delta_\epsilon$. It follows that
\[
m(\lambda_1(X_{g,\epsilon}))=m(\lambda_1(X_g)).
\]
Hence, estimating the upper bound of $m(\lambda_1(X_g))$ is equivalent to estimating that of $m(\lambda_1(X_{g,\epsilon}))$. In the remainder of this section, we will work with $X_{g,\epsilon}$. Moreover, we use the notations 
$$\thick_{g,\epsilon},\ \thin_{g,\epsilon},\ \sA_{k,\epsilon},\ \text{and\ } \collar(\gamma,w)_\epsilon$$
to denote the images of 
$$\thick_g,\ \thin_g,\ \sA_k,\ \text{and\ } \collar(\gamma,w)$$
under the rescaling by $\epsilon^{-2}$, respectively. For any $x\in X_{g,\epsilon}$, the Gaussian curvature 
\begin{equation}\label{Gcur}
    \mathrm{cur}(x)=-\epsilon^2>-1,
\end{equation}
and by \eqref{injlb}, for any $x\in \thick_{g,\epsilon}$, 
\begin{equation}\label{injld1}
    \inj(x)\ge 1.
\end{equation}

Let $r>0$. We recall the notions of $r$-\emph{net} and $r$-\emph{separated set}. 
\begin{definition}\label{def:net-sset}Let $Y \subset X_{g,\epsilon}$ be a subset.
    \begin{enumerate}
        \item A set of points $\{x_1, \cdots, x_l\}$ is called an \emph{$r$-net} of $Y$ if for any $y \in Y$, there exists $x_i$ such that $\dist_\epsilon(x_i, y) \le r$. 
        \item A set of points $\{x_1, \cdots, x_l\}$ is called an \emph{$r$-separated set} if $\dist_\epsilon(x_i,x_j)\ge r$ for any $i\neq j$. 
        \item A set of points $\{x_1, \cdots, x_l\}$ is called an \emph{$r$-separated net} of $Y$ if it is both $r$-separated and an $r$-net of $Y$.
        \end{enumerate}
\end{definition}

Let $\sA_{k,\epsilon}$, $1\leq k\leq I$, be the components of $\thick_{g,\varepsilon}$. For all $r>0$ and $k\in[1,I]$, set
\begin{equation}\label{eqn:sA_kr}
\sA_{k,\epsilon}^r=\{x\in\sA_{k,\epsilon}|\ \dist_\epsilon(x,\thin_{g,\epsilon})\ge r\}.
\end{equation}
Note that $\sA_{k,\epsilon}^r$ might not be connected. 

\begin{lemma}\label{rnet}
    With the same notation as above, 
    if $r\ge 4$ and $\sA_{k,\epsilon}^r\neq\emptyset$, then $\sA_{k,\epsilon}^r$ admits an $r$-separated net of cardinality at most
	\[
	\max \bigg\{1,\, \frac{16\area(\sA_{k,\epsilon})}{\pi\cdot r}\bigg\}.
	\]
\end{lemma}

\begin{proof}
    The proof is standard. Firstly, if there exists $x\in \sA_{k,\epsilon}^r$ such that $\sA_{k,\epsilon}^r\subset B_\epsilon(x,r)$, then the singleton set $\{x\}$ is an $r$-net of $\sA_{k,\epsilon}^r$. Otherwise, let 
    \[ \{x_1,\cdots,x_l\}\subset \sA_{k,\epsilon}^r, \quad l\ge 2 \] 
    be a maximal $r$-separated set. This is clearly an $r$-separated net of $\sA_{k,\epsilon}^r$. So we have 
    \[ B_\epsilon(x_i,r/2)\cap B_\epsilon(x_j,r/2)=\emptyset, \quad \Forall i\neq j. \] 
    Moreover, by the definition of $\sA_{k,\epsilon}^r$, we have $B_\epsilon(x_i, r/2) \subset \sA_{k,\epsilon}$, $\Forall 1 \le i \le l$, so
	\begin{equation}\label{ballAk}
	\sum_{i=1}^l \area\big(B_\epsilon(x_i,r/2)\big)\leq \area(\sA_{k,\epsilon}).
	\end{equation}
	Next, for each $1\le i\le l$, there exists $y_i\in \sA_{k,\epsilon}$ such that $\dist_\epsilon(x_i,y_i)= r/2$. Let $\alpha_i:[0,r/2] \to \thick_{g,\epsilon}$ be a shortest geodesic segment from $x_i$ to $y_i$ parameterized by the arc length. Then for each integer $n \in \{0,\cdots, \lfloor r/2 \rfloor - 1\}$, we have
	\[
	B_\epsilon(\alpha_i(n),1/2)\subset B_\epsilon(x_i, r/2) \subset \sA_{k,\epsilon}, 
	\]
	and these geodesic balls are pairwise disjoint. By \eqref{Gcur} and \eqref{injld1}, $\mathrm{cur}(\alpha_i(n)) = -\epsilon^2$ and $\inj(\alpha_i(n)) \ge 1$, which implies
	\[
	\area\big(B_\epsilon(\alpha_i(n),1/2)\big)= \area\big(B_{\H_\epsilon}(1/2)\big)=\frac{2\pi}{\epsilon^2}(\cosh\frac\epsilon2 -1)\ge \frac\pi 4, 
	\]
	where $B_{\mathbb H_\epsilon}(1/2)$ denotes the geodesic ball of radius $1/2$ in $\H_\epsilon$. The last inequality follows from the estimate $\cosh\epsilon/2 - 1 \ge \epsilon^2/8$. Thus, 
    \begin{align*}
    	\area\big(B_\epsilon(x_i,r/2)\big)
        &\ge \sum_{n=0}^{\lfloor r/2\rfloor - 1}\area\big(B_\epsilon(\alpha_i(n),1/2)\big) \\
        &\ge \frac \pi4\cdot \Big\lfloor \frac r2\Big\rfloor\ge \frac\pi{16}\cdot r.
    \end{align*}
	Combining this with \eqref{ballAk}, we obtain the desired bound. 
\end{proof}

Write $X_{g,\epsilon}= \H_\epsilon/\Gamma$, and let $K_{t,\epsilon}(\cdot,\cdot)$ denote the heat kernel on $X_{g,\epsilon}$. Then for any $f\in L^2(X_{g,\epsilon})$, 
\[ (e^{-t\Delta_\epsilon}f)(x) = \int_{X_{g,\epsilon}}K_{t,\epsilon}(x,y)f(y) \dif y. \] 
Moreover, let $k_{t,\epsilon}(\cdot,\cdot)$ denote the heat kernel on $\H_\epsilon$, then 
\begin{equation}\label{eqn:kernel}
    K_{t,\epsilon}(x,y) = \sum_{\gamma\in\Gamma}k_{t,\epsilon}(\gamma\tilde{x},\tilde{y}), 
\end{equation}
where $\tilde{x}$ and $\tilde{y}$ are lifts of $x$ and $y$, respectively. Since $k_{t,\epsilon}(\tilde{x},\tilde{y})$ depends only on $\dist_{\H_\epsilon}(\tilde{x},\tilde{y})$, we also write it as $k_{t,\epsilon}(\dist_{\H_\epsilon}(\tilde{x},\tilde{y}))$. Next, we present a standard estimate for the heat kernel on $X_{g,\epsilon}$ which will be used in the subsequent analysis.
\begin{lemma}\label{lem:heatest}
    There exists a universal constant $C_1>0$, independent of $\epsilon$, such that for any point $x \in X_{g,\epsilon}$ with $\inj(x)\ge 1$ and any $t \ge 1$, 
    \[
    \|K_{t,\epsilon}(x,\cdot)\|_{L^\infty(X_{g,\epsilon})}\le C_1\exp(4t).
    \]
\end{lemma}
\begin{proof}
    Since
    $$k_{t,\epsilon}(\tilde x,\tilde y)=k_{t,\epsilon}(\dist_{\H_\epsilon}(\tilde x,\tilde y))=\epsilon^2 k_{\epsilon^2 t,1}(\dist_{\H_\epsilon}(\tilde x,\tilde y)),$$
    it follows from \cite[Lemma 3.13]{Bergeron2016} that for all $t\geq1$, 
    \begin{equation}\label{eqn:kernelbound}
        k_{t,\epsilon}(\tilde{x},\tilde{y})\ll \exp(-\frac{\dist_{\H_\epsilon}(\tilde x,\tilde y)^2}{8\epsilon^2t})\ll \exp(-\frac{\dist_{\H_\epsilon}(\tilde x,\tilde y)^2}{8t}) , \quad \Forall \tilde{x}, \tilde{y}\in\H_\epsilon. 
    \end{equation}
    For any $m\geq0\in \N$ and any $\tilde{x}, \tilde{y}\in\H_\epsilon$, since $\inj(x)\ge 1$,
    \begin{equation}\label{eqn:count}
        \begin{aligned}
            &\#\{\gamma\in\Gamma |\ m\leq \dist_{\H_\epsilon}(\gamma\tilde{x}, \tilde{y})<m+1 \}\leq \frac{ \area\big( B_{\H_\epsilon}( m+1+1/2 ) \big) }{\area\big( B_{\H_\epsilon}( 1/2 ) \big)}\\
            &\ll \area\big( B_{\H_\epsilon}( m+1+1/2 ) \big)\ll \area\big( B_{\H}( m+1+1/2)\big)\ll \exp(m). 
        \end{aligned}
    \end{equation}
    Therefore, by \eqref{eqn:kernel}, \eqref{eqn:kernelbound} and \eqref{eqn:count}, 
    \begin{align*}
         K_{t,\epsilon}(x,y) 
        &\ll \sum_{m=0}^\infty \exp(m-\frac{m^2}{8t})\\
        &\leq (4t+1)\exp(2t) + \sum_{m\geq 4t} \exp(8t-m)\\
        &\ll \exp(4t). 
    \end{align*}
    The proof is complete. 
\end{proof}

For simplicity, we write $\lambda_{1,\epsilon}$ for $\lambda_1(X_{g,\epsilon})=\epsilon^2\lambda_1(X_g)$, and omit $X_g$ in $\lambda_1(X_g)$, $\sN_\epsilon(X_g)$, and $N_\epsilon(X_g)$ for the remainder of this section. Set 
\begin{equation}\label{r1r2}
	r_1=c\log\log\Big(\frac{K g}{N_\epsilon+1}\Big),\qquad r_2=c\log\Big(\frac{K g}{N_\epsilon+1}\Big),
\end{equation}
where $c > 0$ is a universal constant that will be chosen later. We assume that $K \geq 10$ is a sufficiently large universal constant such that $\frac{K g}{N_\epsilon + 1}$ is large enough for our analysis, particularly to ensure that $r_1\gg1$ and $r_2/r_1 \gg 1$. 

\subsection{Reduce the problem to the thick part}
In this subsection, we show that it suffices to consider only the thick part of $X_{g,\epsilon}$. Set 
\begin{equation*}
    X_{g,\epsilon}^{r_1} := \{x\in \thick_{g,\epsilon}|\ \dist_{\epsilon}(x,\thin_{g,\epsilon})\ge r_1\}. 
\end{equation*}
First, we have the following lemma. 
\begin{lemma}\label{cardr1}
    There exists an $r_1$-net $\{x_1,\cdots,x_l\}$ of $X_{g,\epsilon}^{r_1}$ such that 
    \[ l \ll \frac{1}{\epsilon^2}\frac g{\log\log \frac{10g}{N_\epsilon+1}}. \]
\end{lemma}
\begin{proof}
    By definition we have 
    \[ X_{g,\epsilon}^{r_1} = \sA_{1,\epsilon}^{r_1}\sqcup\sA_{2,\epsilon}^{r_1}\sqcup\cdots\sqcup\sA_{I,\epsilon}^{r_1},\quad I\geq 1, \]
    where $\sA_{k,\epsilon}^{r_1}$ is defined by \eqref{eqn:sA_kr}. Clearly, the union of $r_1$-nets of $\sA_{k,\epsilon}^{r_1}$'s is an $r_1$-net of $X_{g,\epsilon}^{r_1}$. Thus by Lemma \ref{rnet},
    \begin{align*}
        l&\le \sum_{k=1}^I \max \bigg\{1,\frac{16\area(\sA_{k,\epsilon})}{\pi\cdot r_1}\bigg\}
        \le \sum_{k=1}^I \bigg(1+\frac{16\area(\sA_{k,\epsilon})}{\pi\cdot r_1}\bigg)\\
        &\le N_\epsilon+1+\frac{64\pi\epsilon^{-2}(g-1)}{\pi\cdot r_1}\ll \frac{1}{\epsilon^2}\frac g{\log\log \frac{10g}{N_\epsilon+1}},
    \end{align*}
    where we use the bounds
    \[
    I\le N_\epsilon+1 \text{\quad and\quad} \sum_{k=1}^I\area(\sA_{k,\epsilon})\le \area(X_{g,\epsilon})=4\pi\epsilon^{-2}(g-1).
    \]
    The proof is complete. 
\end{proof}

Next, we extend $\{x_1,\cdots,x_l\}$ to a maximal 1-separated set of $X_{g,\epsilon}$, denoted by
\[
\{x_1,\cdots,x_l,x_{l+1},\cdots, x_s\}
\]
where $l<s\in \mathbb{Z}^{+}$. For each $i\in \{1,\cdots,s\}$, define
\[
V_i:=\big\{ y\in X_{g,\epsilon}|\ \dist_\epsilon(y,x_i)\le \dist_\epsilon(y, x_j),\ \forall j\in\{1,\cdots, s\} \big\}.
\]
Since $\{x_1,\cdots,x_s\}$ is a 1-separated set, 
\[
B_\epsilon(x_i,1/2)\subset V_i\subset B_\epsilon(x_i,1),\quad \forall i\in\{1,\cdots, s\}.
\]
For each $i\in \{1,\cdots,l\}$, set 
\[
\psi_i=\frac 1{\sqrt{\area(V_i)}} \textbf{1}_{\mathring{V}_i},
\]
where $\textbf{1}_{\mathring{V}_i}$ denotes the indicator function of the interior $\mathring{V}_i$ of $V_i$. Then $\norm{\psi_i}_{L^2}=1$. Moreover, $\Forall 1\leq i\leq l$, since $\inj(x_i) \ge 1$, it follows that
\begin{equation}\label{areaVi}
    \area(V_i)\ge \area\big(B_\epsilon(x_i,1/2)\big)=\frac{2\pi}{\epsilon^2}(\cosh\frac\epsilon2 -1)\ge \frac\pi 4. 
\end{equation}
Let $P$ be the orthogonal projection from $L^2(X_{g,\epsilon})$ to $\mathrm{span}\{\psi_1,\cdots,\psi_l\}^{\bot}$, \ie 
\[
Pf=f-\sum_{k=1}^l \langle f,\psi_k\rangle \psi_k,\quad \forall f\in L^2(X_{g,\epsilon}).
\]
Let $\delta_x$ denote the Dirac delta distribution at $x$, then $P\delta_x$ is given as
\[
P\delta_x=\delta_x-\sum_{k=1}^l \psi_k(x)\psi_k.
\] 

Following the ideas of \cite{LM2024, JTYZZ2021}, we will consider the multiplicity of the eigenvalue $e^{-r_1\lambda_{1,\epsilon}}$ in the spectrum of the operator $Pe^{-r_1\Delta_\epsilon}P$. First, we have 
\begin{lemma}\label{lem:subspace}
    If $m(\lambda_{1,\epsilon}) > l + 4N_\epsilon$, then there exists a subspace $S \subset E(\lambda_{1,\epsilon})$ with 
    \begin{equation*}
        \dim S \ge m(\lambda_{1,\epsilon})-l-4N_\epsilon,
    \end{equation*}
    such that each $\varphi\in S$ is an eigenfunction of $Pe^{-r_1\Delta_\epsilon}P$ with eigenvalue $e^{-r_1\lambda_{1,\epsilon}}$, and 
    \begin{equation}\label{intcont}
    	\int_{X_{g,\epsilon}}\varphi^2 \mathrm dA\leq 3 \int_{\thick_{g,\epsilon}}\varphi^2 \mathrm dA.
    \end{equation}
\end{lemma}
\begin{proof}
    By Lemma \ref{lem:mass}, there exists a subspace $S_1\subset E(\lambda_1)$ with 
    \[
    \dim S_1\ge m(\lambda_1) - 4(\lfloor\sqrt\lambda_1\rfloor+1)N_\epsilon = m(\lambda_1) - 4N_\epsilon,
    \]
    such that every $\varphi\in S_1$ satisfies the inequality $$\int_{X_{g}}\varphi^2 \mathrm dA\leq 3 \int_{\thick_{g}}\varphi^2 \mathrm dA.$$ 
    Since each $\varphi \in E(\lambda_1)$ corresponds to an eigenfunction $\varphi_\epsilon \in E(\lambda_{1,\epsilon})$ under the rescaling by $\epsilon^{-2}$, we obtain a subspace $S_{1,\epsilon} \subset E(\lambda_{1,\epsilon})$ with
    $$\dim S_{1,\epsilon}=\dim S_1\ge m(\lambda_1) - 4N_\epsilon,$$
    such that every function in $S_{1,\epsilon}$ satisfies the inequality \eqref{intcont}.
    
    Next, consider the orthogonal projection
    \[
    \Pi:S_{1,\epsilon}\to \mathrm{Span}\{\psi_1,\cdots,\psi_l\}.
    \]
    Then 
    \[
    \dim(\ker \Pi) \ge\dim S_{1,\epsilon}-l\ge m(\lambda_1) - l - 4N_\epsilon,
    \]
    and every $\varphi \in \ker \Pi$ satisfies
    \[
    P\varphi=(\mathrm{Id}-\Pi)\varphi=\varphi,
    \]
    which implies
    \[
    Pe^{-r_1\Delta_\epsilon}P\varphi=Pe^{-r_1\Delta_\epsilon}\varphi=P(e^{-r_1\lambda_{1,\epsilon}}\varphi)=e^{-r_1\lambda_{1,\epsilon}}\varphi.
    \]
    Therefore, the subspace $\ker \Pi$ satisfies the desired properties.
\end{proof}

Let $\{\varphi_j\}_{j=1}^\infty$ be an orthonormal basis of $L^2(X_{g,\epsilon})$ consisting of eigenfunctions of the self-adjoint compact operator $Pe^{-r_{1,\epsilon}\Delta}P$, with the associated eigenvalues $\{\Lambda_j\}_{j=1}^\infty$. According to Lemma \ref{lem:subspace}, one may choose this basis such that
\begin{equation}\label{eqn:basis}
    \mathrm{Span}\{\varphi_1,\cdots, \varphi_{\dim S}\}=S \text{\quad and\quad} \Lambda_1=\cdots=\Lambda_{\dim S}=e^{-r_1\lambda_{1,\epsilon}}. 
\end{equation}
To give an upper bound of $m$, it suffices to give an upper bound of $\dim S$. Set 
\begin{equation*}
    m' = \dim S \text{\quad and\quad} n = \lfloor r_2/r_1\rfloor. 
\end{equation*} 
We have the following two lemmas. 
\begin{lemma}\label{lem:norm}
    Let $\delta_x$ denote the Dirac delta distribution at $x$. Then 
    \begin{equation*}
        \norm{(Pe^{-r_1\Delta_\epsilon}P)^{n+1}\delta_x}^2_{L^2(X_{g,\epsilon})} = \sum_{j=1}^{\infty} \Lambda_j^{2n+2}\varphi_j^2. 
    \end{equation*}
\end{lemma}
\begin{proof}
    The claim is based on a direct computation. For any $f\in L^2(X_{g,\epsilon})$, suppose that 
    \[ f = \sum_{j=1}^\infty a_j\varphi_j,\quad a_j\in\R, \] 
    then since $Pe^{-r_1\Delta_\epsilon}P$ is self-adjoint, 
    \begin{align*}
        &\big\langle (Pe^{-r_1\Delta_\epsilon}P)^{n+1}\delta_x,\ f\big\rangle
         = ((Pe^{-r_1\Delta_\epsilon}P)^{n+1}f)(x)\\
        & = \sum_{j=1}^\infty a_j\Lambda_j^{n+1}\varphi_j(x) = \Big\langle \sum_{j=1}^\infty \Lambda_j^{n+1}\varphi_j(x)\varphi_j,\ f\Big\rangle.
    \end{align*}
    Thus, we have 
    \begin{equation*}
        \norm{(Pe^{-r_1\Delta_\epsilon}P)^{n+1}\delta_x}^2_{L^2(X_{g,\epsilon})}
         = \Big\| \sum_{j=1}^{\infty}\Lambda_j^{n+1}\varphi_j(x)\varphi_j \Big\|^2_{L^2(X_{g,\epsilon})}
         = \sum_{j=1}^{\infty} \Lambda_j^{2n+2}\varphi_j^2, 
    \end{equation*}
    as desired. 
\end{proof}
\begin{lemma}\label{1mul}
    With the same notation as above, 
    \begin{equation*}
        m'\cdot e^{-r_1\lambda_{1,\epsilon}(2n+2)} \leq 3  \int_{\thick_{g,\epsilon}}  \norm{(Pe^{-r_1\Delta_\epsilon}P)^{n+1}\delta_x}^2_{L^2(X_{g,\epsilon})} \mathrm dA_x. 
    \end{equation*}
\end{lemma}
\begin{proof}
    By Lemma \ref{lem:subspace} and Lemma \ref{lem:norm}, we have 
    \begin{align*}
        m'\cdot e^{-r_1\lambda_{1,\epsilon}(2n+2)}
        & = e^{-r_1\lambda_{1,\epsilon}(2n+2)}\int_{X_{g,\epsilon}} \sum_{j=1}^{m'} \varphi_j^2 \dif A \\
        &\le 3 e^{-r_1\lambda_{1,\epsilon}(2n+2)}\int_{\thick_{g,\epsilon}} \sum_{j=1}^{m'} \varphi_j^2 \dif A \qquad\text{(by \eqref{intcont})}\\
        &\le 3  \int_{\thick_{g,\epsilon}} \sum_{j=1}^{\infty} \Lambda_j^{2n+2}\varphi_j^2 \dif A \qquad\text{(by \eqref{eqn:basis})}\\
        & = 3  \int_{\thick_{g,\epsilon}} \norm{(Pe^{-r_1\Delta_\epsilon}P)^{n+1}\delta_x}^2_{L^2(X_{g,\epsilon})} \dif A_x.
    \end{align*}
    The proof is complete. 
\end{proof}

\subsection{Estimate of the integral on the thick part}
In this subsection, we give an upper bound for 
\[
\int_{\thick_{g,\epsilon}}  \norm{(Pe^{-r_1\Delta_\epsilon}P)^{n+1}\delta_x}^2_{L^2(X_{g,\epsilon})} \mathrm dA_x.
\]
We decompose $X_{g,\epsilon}^{\mathrm {thick}}$ into the following two disjoint sets (see Figure \ref{fig:Y1Y2}): 
\begin{itemize}
	\item $Y_1:=\{x\in\thick_{g,\epsilon}|\ \dist_{\epsilon}(x,\thin_{g,\epsilon})\le r_1 + C'r_2 + 2\}$;
	\item $Y_2:=X_{g,\epsilon}^{\mathrm {thick}}\setminus Y_1$,
\end{itemize}
where $C' > 0$ is a universal constant that will be fixed later. 
\begin{figure}
    \centering
    \includegraphics[scale = 0.28]{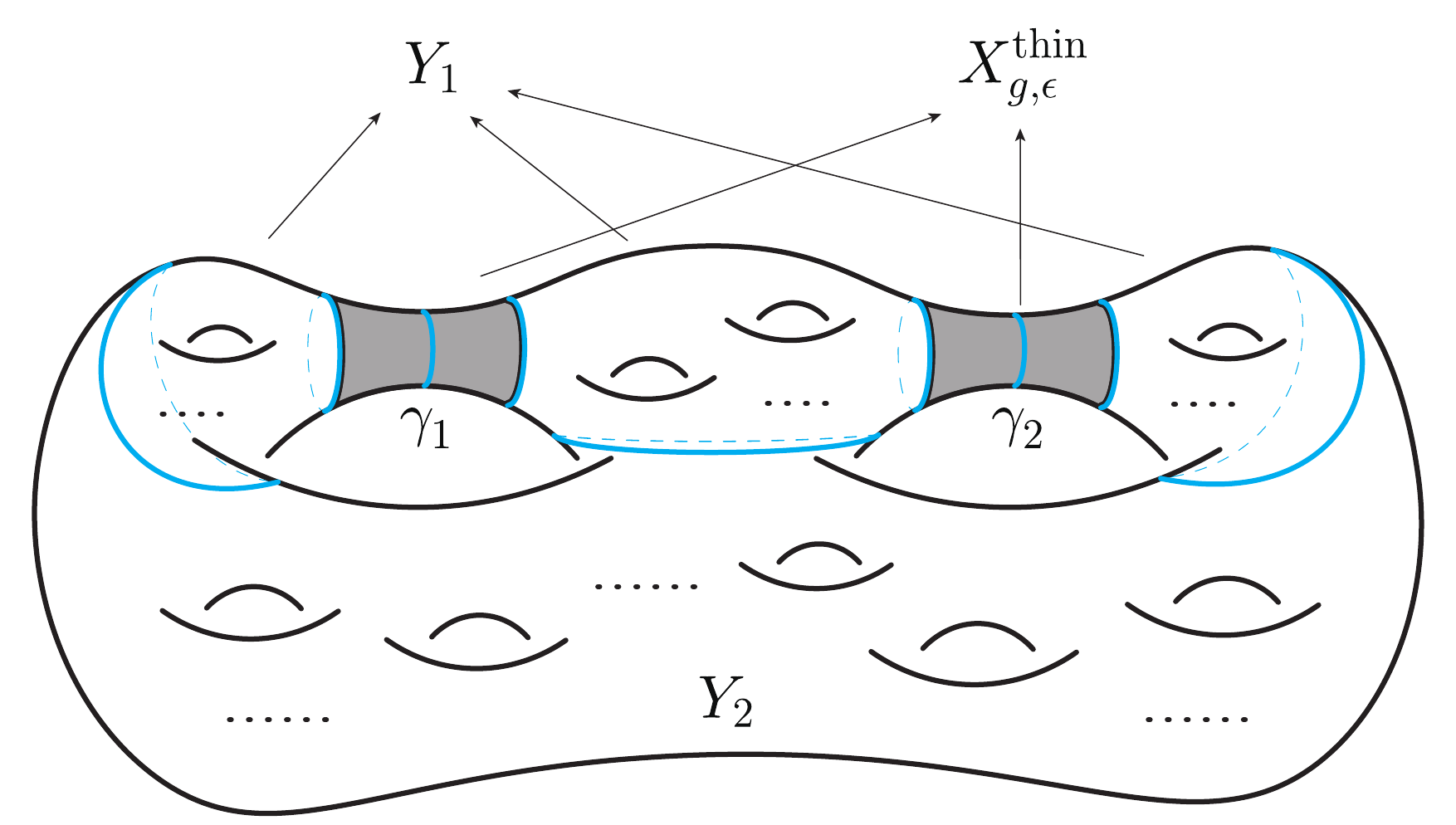}
    \caption{An illustration for $Y_1$ and $Y_2$}
    \label{fig:Y1Y2}
\end{figure}

\subsubsection{Estimate on $Y_1$}
We first give an upper bound for 
\[ \int_{Y_1}  \norm{(Pe^{-r_1\Delta_\epsilon}P)^{n+1}\delta_x}^2_{L^2(X_{g,\epsilon})} \mathrm dA_x. \]
We have the following two lemmas. 
\begin{lemma}\label{lem:normbound}
    For any $x\in Y_1$, 
    \[ \norm{(Pe^{-r_1\Delta_\epsilon}P)^{n+1}\delta_x}_{L^2(X_{g,\epsilon})} \ll 1. \] 
\end{lemma}
\begin{proof}
    Since the operator norms of $P$, $e^{-r_1\Delta_\epsilon}$ and $e^{-(r_1-1)\Delta_\epsilon}$ are all no more than $1$, we have
    \begin{equation}\label{norm01}
         \|(Pe^{-r_1\Delta_\epsilon}P)^{n+1}\delta_x\|_{L^2(X_{g,\epsilon})} \le\|e^{-\Delta_\epsilon}P\delta_x\|_{L^2(X_{g,\epsilon})}.   
    \end{equation}
    Let $k_x\in\{1,\cdots,l\}$ be the unique index such that $x\in \mathring{V}_{k_x}$, if such an index exists. In this case, we have 
    \begin{equation}\label{norm1}
        \begin{aligned}
            (e^{-\Delta_\epsilon}P\delta_x)(y)
            & = e^{-\Delta_\epsilon}\delta_x - \psi_{k_x}(x)\cdot (e^{-\Delta_\epsilon}\psi_{k_x})(y)\\
            & = K_{1,\epsilon}(x,y)-\psi_{k_x}(x)\int_{X_{g,\epsilon}}K_{1,\epsilon}(y,z)\psi_{k_x}(z)\mathrm dA_z. 
        \end{aligned}
    \end{equation}
    Then by \eqref{areaVi}, \eqref{norm1} and Lemma \ref{lem:heatest}, we have
    \begin{equation}\label{estxinY1}
    	\begin{aligned}
    		  &\|e^{-\Delta_\epsilon}P\delta_x\|_{L^2(X_{g,\epsilon})} \\
    	 \le{} &\|K_{1,\epsilon}(x,\cdot)\|_{L^2(X_{g,\epsilon})}+\|\psi_{k_x}\|_{L^\infty(X_{g,\epsilon})}\cdot\|e^{-\Delta_\epsilon}\psi_{k_x}\|_{L^2(X_{g,\epsilon})} \\
    	 \le{} &\|K_{1,\epsilon}(x,\cdot)\|_{L^\infty(X_{g,\epsilon})}^{\frac 12}\cdot\|K_{1,\epsilon}(x,\cdot)\|_{L^1(X_{g,\epsilon})}^{\frac 12}+\frac{2}{\sqrt{\pi}} \\
    	 \le{} &\sqrt C_1+\frac{2}{\sqrt{\pi}}.
    	\end{aligned}
    \end{equation}
    If no such $k_x$ exists, then $e^{-\Delta_\epsilon}P\delta_x=e^{-\Delta_\epsilon}\delta_x$, and hence
    \begin{equation}\label{notestxinY1}
        \|e^{-\Delta_\epsilon}P\delta_x\|_{L^2(X_{g,\epsilon})}=\|K_{1,\epsilon}(x,\cdot)\|_{L^2(X_{g,\epsilon})}\le \sqrt C_1. 
    \end{equation}
    The lemma then follows from \eqref{norm01}, \eqref{norm1} \eqref{estxinY1} and \eqref{notestxinY1}. 
\end{proof}

\begin{lemma}\label{contY1}
Under the above assumptions,
\[
    \begin{aligned}
        &\int_{Y_1} \|(Pe^{-r_1\Delta_\epsilon}P)^{n+1}\delta_x\|_{L^2(X_{g,\epsilon})}^2\mathrm dA_x \ll \frac{N_\epsilon\exp(2C'r_2\epsilon)}{\epsilon^2}. 
    \end{aligned}
\]
\end{lemma}
\begin{proof}
    For each $\gamma\in \sN_\epsilon$, let $x_{\gamma,\epsilon}$ and $y_{\gamma,\epsilon}$ be two points in distinct components of $\partial \collar(\gamma, w(\gamma)-1)_\epsilon$. Since by \eqref{eqn:lw-1} and rescaling, the length of each component of $\partial \collar(\gamma, w(\gamma)-1)_\epsilon$ is less than $2\epsilon^{-1}$, we have 
    \[
    Y_1\subset \bigcup_{\gamma\in \sN_\epsilon} \big(B_\epsilon(x_{\gamma,\epsilon},r_1+C'r_2+2+\epsilon^{-1})\cup B_\epsilon(y_{\gamma,\epsilon},r_1+C'r_2+2+\epsilon^{-1})\big).
    \]
    Denote $R_\epsilon=r_1+C'r_2+2+\epsilon^{-1}$, then, 
    \begin{equation}\label{areaY1}
        \begin{aligned}
            \area(Y_1)
            &\le \sum_{\gamma\in \sN_\epsilon}\big(\area(B_\epsilon(x_{\gamma,\epsilon},R_\epsilon))+\area(B_\epsilon(y_{\gamma,\epsilon},R_\epsilon))\big)\\
            &\le 2N_\epsilon\cdot \area(B_{\mathbb H_\epsilon}(2C'r_2+3\epsilon^{-1}))\\
            &=   2N_\epsilon\cdot\frac{2\pi}{\epsilon^2}(\cosh(2C'r_2\epsilon+3)-1) \\
            &\ll N_\epsilon\epsilon^{-2}\exp(2C'r_2\epsilon), 
        \end{aligned}
    \end{equation}
    where $B_{\mathbb H_{\epsilon}}(r)$ denotes the geodesic ball of radius $r$ in $\mathbb H_\epsilon$.
    The conclusion then follows from \eqref{areaY1} and Lemma \ref{lem:normbound}.
\end{proof}

\subsubsection{Estimate on $Y_2$}
We omit the subscript $L^2(X_{g,\epsilon})$ in the following for simplicity. Following the approach in \cite[Section 3--4]{LM2024} of Letrouit--Machado, we give an upper bound for 
\[
\int_{Y_2}  \|(Pe^{-r_1\Delta_\epsilon}P)^{n+1}\delta_x\|^2 \mathrm dA_x.
\]
For any $x\in Y_2$, let $\chi_x$ be the indicator function of the subset
\begin{equation*}
    \bigcup_{1\le i\le s,V_i\cap B_\epsilon(x,C'r_2)\neq\emptyset}V_i.
\end{equation*}
Note that 
\[
\supp(\chi_x)\subset B_\epsilon(x,C'r_2+2),
\]
by the definition of $Y_2$, this implies that 
\begin{equation*}
\supp(\chi_x)\subset X_{g,\epsilon}^{r_1},\quad \forall x\in Y_2.
\end{equation*}
Thus, every point $y\in \supp(\chi_x)$ satisfies $$\inj(y)\ge1.$$ Recall the following facts:
\begin{itemize}
    \item Since $X_{g,\epsilon}^{r_1} \subset \thick_{g,\epsilon}$, it follows from \eqref{Gcur} and \eqref{injld1} that for all $x \in X_{g,\epsilon}^{r_1}$, we have $\mathrm{cur}(x) > -1$ and $\inj(x) \ge 1$.
    \item $\{x_1,\cdots,x_l\}$ is an $r_1$-net of $X_{g,\epsilon}^{r_1}$, hence is also an $r_1$-net of $Y_2$; 
    \item $\|\psi_i\|_{L^\infty(X_g)} \le 2/\sqrt{\pi}$ for all $1\le i\le l$. 
\end{itemize}
These facts allow us to repeat the arguments in the proofs of \cite[Lemma 3.1, 3.4, 3.8]{LM2024}, taking $\rho = 1$ and $b = -1$ in the initial assumptions of \cite[Section 3]{LM2024}, to derive the following three special cases of those lemmas.

\begin{lemma}[\cite{LM2024}, Lemma 3.1]\label{lem3.1}
    There exists a universal constant $C_2>0$ such that if $C'\ge 48$, then for any $x\in Y_2$,
    \begin{align*}
        &\norm{(Pe^{-r_1\Delta_\epsilon}P)^{n+1}\delta_x}\\
        &\le C_2 \bigg( \sup_{\|\varphi\|=1} \norm{(P\chi_xe^{-r_1\Delta_\epsilon}\chi_xP)^{n}\varphi} + \exp(-\frac{C'^2r_2}{32}) \bigg).
    \end{align*}
\end{lemma}
\noindent Following \cite{LM2024}, we denote by $\varphi_x\in L^2(X_{g,\epsilon})$ a function that attains the following supremum
\[
\sup_{\|\varphi\|=1} \norm{P\chi_xe^{-r_1\Delta_\epsilon}\chi_xP\varphi}.
\]

\begin{lemma}[\cite{LM2024}, Lemma 3.4]\label{lem3.4}
    Assume that $C'^2\ge 80+48C'$. Then there exists a universal constant $C_3>0$, and a subset $Y_3\subset Y_2$ with 
    \[
    \area(Y_3)\ll \exp(3C'r_2),
    \]
    such that for any $x\in Y_2\setminus Y_3$,
    \begin{equation*}
        \norm{e^{-r_1\Delta_\epsilon}|\varphi_x|}^2\le e^{-2r_1\lambda_{1,\epsilon}} + C_3\exp(-C'r_2).
    \end{equation*}
\end{lemma}

\begin{lemma}[\cite{LM2024}, Lemma 3.8]\label{lem3.8}
    There exist two universal constants $C_4,C_5>0$ such that for any $C>16$ and any $x\in Y_2$,
    \begin{align*}
        &\|P\chi_xe^{-r_1\Delta_\epsilon}\chi_xP\varphi_x\|^2\\
        &\le \big(1 - C_4\exp(-hr_1)\big)\cdot \norm{e^{-r_1\Delta_\epsilon}|\varphi_x|}^2+ C_5\exp(r_1-\frac{C^2r_1}{64}),
    \end{align*}
    where $h=50(1+C)$. 
\end{lemma}

We now estimate the integral on $Y_2$. Assume that 
    \begin{align}
       &C'>48,\qquad C>16,\\
       &c<\min\left(\frac 1{8+6C'}, \frac {1}{4h}\right),\label{c}\\
       &C'^2>\max\left(80+48C',\frac{16}{c}\right),\label{C'2}\\
       &\frac{C^2}{64}>h+4,\label{C2}
    \end{align}
for example, $C'=10^5, C=10^4, c=10^{-7}$. 
First, by Lemma \ref{lem3.1} and Lemma \ref{lem3.4}, 
\begin{equation}\label{contY2}
    \begin{aligned}
               &\int_{Y_2}  \|(Pe^{-r_1\Delta_\epsilon}P)^{n+1}\delta_x\|^2 \mathrm dA_x\\
         \ll{} &\int_{Y_2}\norm{(P\chi_xe^{-r_1\Delta_\epsilon}\chi_xP)^{n}\varphi_x}^2 + \exp(-\frac{C'^2r_2}{16}) \mathrm dA_x\\
         \ll{} &\left(\int_{Y_2\setminus Y_3}+\int_{Y_3}\right)\norm{(P\chi_xe^{-r_1\Delta_\epsilon}\chi_xP)^{n}\varphi_x}^2\mathrm dA_x+\epsilon^{-2}g\exp(-\frac{C'^2 r_2}{16})\\
         \ll{} &\int_{Y_2\setminus Y_3}\|(P\chi_xe^{-r_1\Delta_\epsilon}\chi_xP)^{n}\varphi_x\|^2\mathrm dA_x+\exp(3C'r_2)+ \epsilon^{-2}g\exp(-\frac{C'^2 r_2}{16}),
    \end{aligned}
\end{equation}
where we use the bounds 
\[ \area(Y_2)\ll \epsilon^{-2}g \text{\quad and\quad} \|(P\chi_xe^{-r_1\Delta_\epsilon}\chi_xP)^{n}\varphi_x\|\leq 1,\ \Forall x\in Y_2. \]

Next, we use Lemma \ref{lem3.8} to give the following bound. 
\begin{lemma}\label{Y2-Y3}
    Under the above assumptions, if $1\ll r_1\ll r_2$, then 
    \[
        \int_{Y_2\setminus Y_3}\|(P\chi_xe^{-r_1\Delta_\epsilon}\chi_xP)^{n}\varphi_x\|^2\mathrm dA_x \ll \epsilon^{-2}g\cdot \exp(-\frac 14C_4\exp(-hr_1)n)e^{-2r_1n\lambda_{1,\epsilon}}. 
    \]
\end{lemma}
\begin{proof}
    First, if $r_2/r_1$ is sufficiently large , then 
    \[
        C_3\exp(-C'r_2)\leq \frac{1}{2}C_4\exp(-hr_1) e^{-2r_1\lambda_{1,\epsilon}}.
    \] 
    Thus by Lemma \ref{lem3.4}, for any $x \in Y_2 \setminus Y_3$, 
    \[
    \begin{aligned}
       &\big(1-C_4\exp(-hr_1)\big)\cdot \norm{e^{-r_1\Delta_\epsilon}|\varphi_x|}^2\\
  \le{}&\big(1-C_4\exp(-hr_1)\big)\cdot \big(e^{-2r_1\lambda_{1,\epsilon}}+C_3\exp(-C'r_2)\big)\\
  \le{}&\big(1-\frac 12C_4\exp(-hr_1)\big)\cdot e^{-2r_1\lambda_{1,\epsilon}}.
    \end{aligned}
    \]
    Next, by \eqref{C2} and our assumption $\lambda_{1,\epsilon}=\epsilon^2\lambda_1 < 1$, if $r_1$ is sufficiently large, then 
    \[
    \begin{aligned}
        C_5\exp(r_1-\frac{C^2r_1}{64})
        \le C_5\exp\big(-hr_1-3r_1\big)
        \le\frac14 C_4\exp(-hr_1)e^{-2r_1\lambda_{1,\epsilon}}.
    \end{aligned}
    \]
    Combining the above two estimates, it follows from Lemma \ref{lem3.8} that 
    \[
        \|P\chi_xe^{-r_1\Delta_\epsilon}\chi_xP\varphi_x\|^2 \leq \Big(1-\frac 14C_4\exp(-hr_1)\Big)\cdot e^{-2r_1\lambda_{1,\epsilon}}.
    \]
    Finally, by \eqref{c}, we obtain 
    \[
    \begin{aligned}
           & \int_{Y_2\setminus Y_3}\|(P\chi_xe^{-r_1\Delta_\epsilon}\chi_xP)^{n}\varphi_x\|^2\mathrm dA_x
        \\
        ={}& \int_{Y_2\setminus Y_3}\|P\chi_xe^{-r_1\Delta_\epsilon}\chi_xP\varphi_x\|^{2{n}}\mathrm dA_x\\
         \ll{}& \epsilon^{-2}g\cdot \Big(1-\frac 14C_4\exp(-hr_1)\Big)^{n}e^{-2r_1n\lambda_{1,\epsilon}}\\
        \leq{}& \epsilon^{-2}g\cdot \exp(-\frac 14C_4\exp(-hr_1)n)e^{-2r_1n\lambda_{1,\epsilon}},
    \end{aligned}
    \]
    as desired. 
\end{proof}

Now we are ready to prove Theorem \ref{thm:mt-1}.

\begin{proof}[Proof of Theorem \ref{thm:mt-1}]
Combining Lemma \ref{1mul}, Lemma \ref{contY1}, \eqref{contY2} and Lemma \ref{Y2-Y3}, we have
\begin{equation}\label{1upbd}
    \begin{aligned}
        m'\cdot e^{-r_1\lambda_{1,\epsilon}(2n+2)}
  \ll{} &\epsilon^{-2}g\cdot \exp(-\frac 14C_4\exp(-hr_1)n)e^{-2r_1n\lambda_{1,\epsilon}}\\
        &+ \exp(3C'r_2) + \epsilon^{-2}g\exp(-\frac{C'^2r_2}{16}) + \epsilon^{-2}N_\epsilon\exp(2C'r_2\epsilon)\\
  \ll{} &\epsilon^{-2}g\cdot \exp(-\frac 14C_4\exp(-hr_1)n)e^{-2r_1n\lambda_{1,\epsilon}}\\ 
        &+ \epsilon^{-2}g\exp(-\frac{C'^2r_2}{16}) + \epsilon^{-2}(N_\epsilon+1)\exp(3C'r_2).
    \end{aligned}
\end{equation}
To obtain an upper bound for $m'$, it suffices to control the following three terms:
\begin{itemize}
    \item $e^{r_1\lambda_{1,\epsilon}(2n+2)}\cdot \epsilon^{-2}(N_\epsilon+1)\exp(3C'r_2)$,
    \item $e^{r_1\lambda_{1,\epsilon}(2n+2)}\cdot \epsilon^{-2}g \exp(-\frac{C'^2r_2}{16})$,
    \item $e^{r_1\lambda_{1,\epsilon}(2n+2)}\cdot \epsilon^{-2}g \exp(-\frac 14C_4\exp(-hr_1)n)e^{-2r_1n\lambda_{1,\epsilon}}$.
\end{itemize}
Recall that $\lambda_{1,\epsilon} < 1$ and $ n = \lfloor r_2/r_1 \rfloor$. It follows from \eqref{c} that 
\begin{equation}\label{1term}
\begin{aligned}
    &e^{r_1\lambda_{1,\epsilon}(2n+2)}\cdot \epsilon^{-2}(N_\epsilon+1)\exp(3C'r_2)
    \le \epsilon^{-2}(N_\epsilon+1)\exp((4+3C')r_2)\\
    ={}&\epsilon^{-2}(N_\epsilon+1) \Big(\frac {Kg}{N_\epsilon+1}\Big)^{(4+3C')c} \ll \epsilon^{-2}(g(N_\epsilon+1))^{\frac 12}.
\end{aligned}
\end{equation}
Similarly, by \eqref{c} and \eqref{C'2}, 
\begin{equation}\label{2term}
\begin{aligned}
     &e^{r_1\lambda_{1,\epsilon}(2n+2)}\cdot \epsilon^{-2}g\exp(-\frac{C'^2r_2}{16})
    \le \epsilon^{-2}g\exp(\Big(4-\frac{C'^2}{16}\Big)r_2)\\
    ={}&\epsilon^{-2}g\Big(\frac{Kg}{N_\epsilon+1}\Big)^{\big(4-\frac{C'^2}{16}\big)c} \le \epsilon^{-2}g\Big(\frac{Kg}{N_\epsilon+1}\Big)^{4c-1} \ll\epsilon^{-2} (g(N_\epsilon+1))^{\frac 12}.
\end{aligned}
\end{equation}
For the third term, if $1\ll r_1\ll r_2$, then by \eqref{c}, 
\begin{equation}\label{3term}
\begin{aligned}
         &e^{r_1\lambda_{1,\epsilon}(2n+2)} \epsilon^{-2}g\exp(-\frac 14C_4\exp(-hr_1)n) e^{-2r_1n\lambda_{1,\epsilon}}\\
   \leq{}&\epsilon^{-2}g\cdot \exp(2r_1 - \frac 14C_4\exp(-\frac{1}{4c}r_1) n)\\
   \ll{} &\epsilon^{-2}g\cdot \exp(-\frac 15C_4\Big(\log\frac{Kg}{N_\epsilon+1}\Big)^{\frac 23})\\
   \ll{} &\epsilon^{-2}\frac {g}{\log \frac {Kg}{N_\epsilon+1}},
\end{aligned}
\end{equation}
where in the second inequality we apply 
\[\Big(\log\frac{Kg}{N_\epsilon+1}\Big)^{\frac 23}\ll n\exp(-\frac{1}{4c}r_1) \text{\quad and\quad} 2r_1\ll \frac 1{20}C_4\Big(\log\frac{Kg}{N_\epsilon+1}\Big)^{\frac 23}.\]

\noindent Finally, combining \eqref{1upbd}, \eqref{1term}, \eqref{2term}, \eqref{3term}, Lemma \ref{cardr1} and Lemma \ref{lem:subspace}, we obtain 
\[
\begin{aligned}
    m &\ll l+4N_\epsilon+\epsilon^{-2}(g(N_\epsilon+1))^{\frac 12}+\epsilon^{-2}\frac {g}{\log \frac {Kg}{N_\epsilon+1}}\\
      &\ll \epsilon^{-2}\frac {g}{\log\log \frac{10g}{N_\epsilon+1}}+N_\epsilon+\epsilon^{-2}(g(N_\epsilon+1))^{\frac 12}+\epsilon^{-2}\frac {g}{\log \frac {10g}{N_\epsilon+1}}\\
      &\ll \epsilon^{-2}\frac {g}{\log\log \frac{10g}{N_\epsilon+1}}.
\end{aligned}
\]
This completes the proof of Theorem \ref{thm:mt-1}.
\end{proof}

\section{Proof of Theorem \ref{thm:mt-2}}\label{section5}
In this section, we prove Theorem \ref{thm:mt-2}.

We use a thick-thin decomposition different from that used in the proof of Theorem \ref{thm:mt-1}. For any $\delta\in(0,1/2)$, let 
\[ \epsilon(\delta) = \arcsinh\frac{1}{\sinh(1/\delta+2)}. \] 
Then for any $\epsilon\in(0,\epsilon(\delta))$, we define 
\begin{equation}\label{eqn:delta-thickthin}
    \begin{aligned}
        \thin_g &\df \bigcup_{\gamma\in \sN_{\epsilon}(X_g)}\collar(\gamma,w(\gamma)-\frac{1}{\delta}),\\ 
        \thick_g &\df \overline{X\setminus\thin_g} = \sA_1\sqcup\sA_2\sqcup\cdots\sqcup\sA_{I_\epsilon}, \quad I_\epsilon\geq1. 
    \end{aligned}
\end{equation}
Next, for each $\gamma\in \sN_{\epsilon}(X_g)$, we define \emph{shell} of $\collar(\gamma,w(\gamma)-1/\delta)$ as 
\begin{equation*}
    \sS(\gamma) := \collar(\gamma,w(\gamma)-\frac{1}{\delta})\setminus \collar(\gamma,w(\gamma)-\frac{1}{\delta}-2). 
\end{equation*}
Let $\sS_k$ denote the union of the connected components of the shells adjacent to $\sA_k$. 

Since $\sA_k$ is connected, for any $p,q\in\sA_k$, there exists a shortest geodesic segment joining them that does not intersect with any $\gamma\in \sN_{\epsilon}(X_g)$. 
We have the following lemma: 
\begin{lemma}\label{lem:inj}
    Suppose $p,q\in\sA_k$, and let $\alpha$ be a shortest geodesic segment joining $p$ and $q$ that does not intersect with any $\gamma\in \sN_{\epsilon}(X_g)$. Then for any $x\in \alpha$, 
    \begin{equation*}
        \inj(x)>\epsilon \text{\quad and\quad} B(x,\epsilon)\subset\sA_k\cup\sS_k. 
    \end{equation*}
\end{lemma}
\begin{proof}
    For any $x\in \alpha$, if $x\notin \collar(\gamma,w(\gamma))$ for all $\gamma\in \sN_{\epsilon}(X_g)$, then the lemma follows by definition. If $x\in \collar(\gamma,w(\gamma))$ for some $\gamma\in \sN_{\epsilon}(X_g)$, by considering the restriction of $\alpha$ to $\thin_g$, one may assume that both $p$ and $q$ are in one component of the boundary $\partial \collar(\gamma,w(\gamma)-1/\delta)$. See Figure \ref{fig:collar} for an illustration. 
    \begin{figure}
        \centering 
        \includegraphics[scale = 0.28]{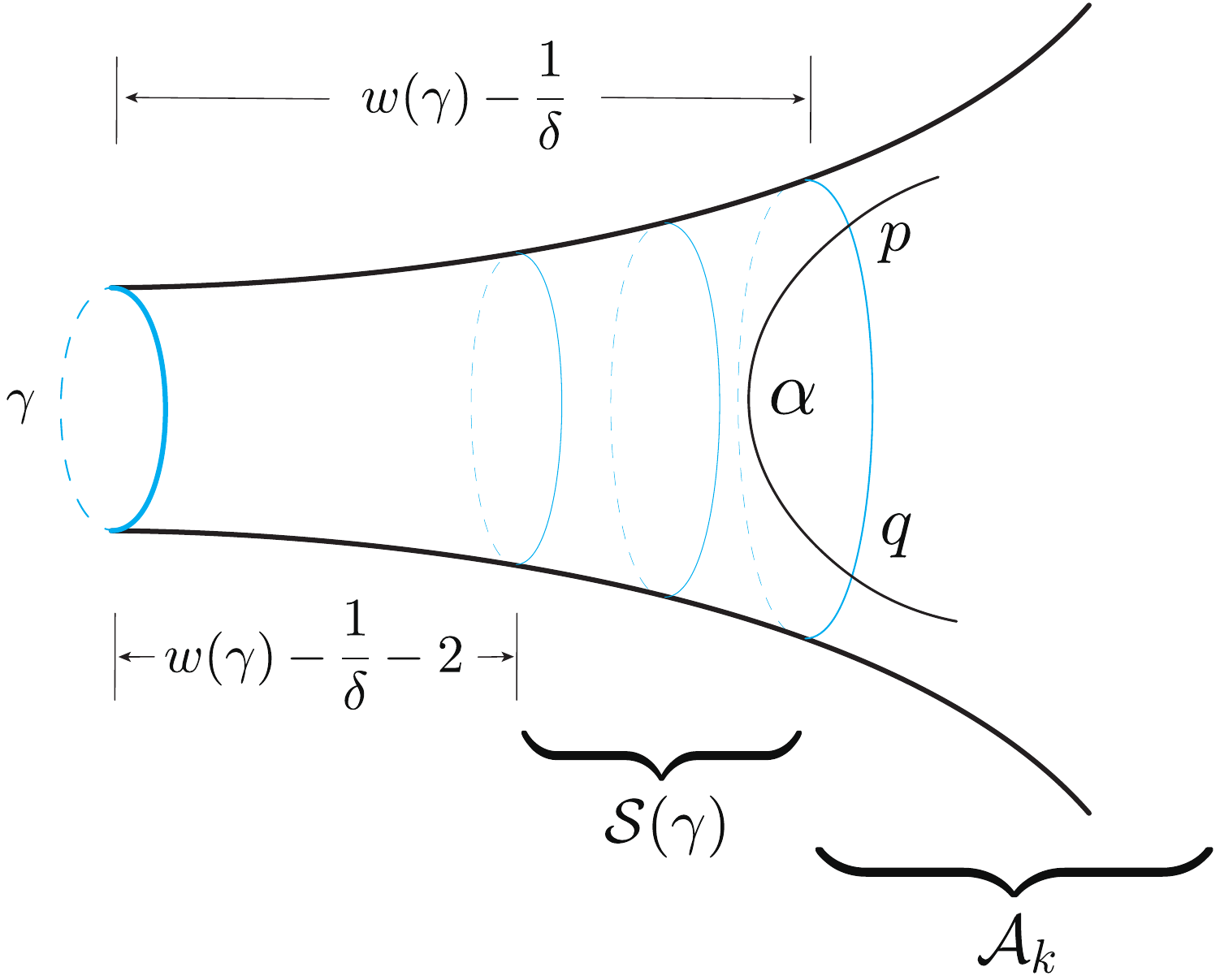}
        \caption{An illustration for the position of $\alpha$.}
        \label{fig:collar}
    \end{figure}
    We first claim that
    \begin{equation}\label{d-i-b}
     d:=\dist\big(x,\partial \collar(\gamma,w(\gamma))\big)\leq 1/\delta+1.   
    \end{equation} 
    Otherwise, $\ell(\alpha)\geq 2$, while by \eqref{eqn:lw-1}, $\ell(\alpha_{pq})<2$, which leads to a contradiction since we assume that $\alpha$ is the shortest. Thus, $d\leq 1/\delta+1$, and by \eqref{eqn:inj} we have 
    \begin{align*}
        \inj(x)
         & =  \arcsinh\Big(\cosh\frac{\ell(\gamma)}{2}\cosh d-\sinh d\Big)\\
         &\geq \arcsinh(\cosh d-\sinh d)\\
        &\geq \arcsinh\frac{1}{e^{1/\delta+1}}\\
        &  >  \arcsinh\frac{1}{\sinh(1/\delta+2)} = \epsilon(\delta) > \epsilon. 
    \end{align*}
    Moreover, since $\epsilon<1$, by \eqref{d-i-b} we have $B(x,\epsilon)\subset\sA_k\cup\sS_k$, this concludes the proof of the lemma. 
\end{proof}

Let $f$ be a normalized eigenfunction of $\Delta$ with eigenvalue $\lambda<1/4$, then we have the following lemma, which estimates the oscillations of $f$ over the components of the thick part of $X_g$. 
\begin{lemma}\label{lem:oscillation}
    Suppose $p\in\sA_k$, $q\in\sA_k\cup\sS_k$ and let $\alpha$ be a shortest geodesic segment joining $p$ and $q$ that does not intersect with any $\gamma\in \sN_{\epsilon}(X_g)$. Then 
    \begin{equation*}
        \abs{f(p)-f(q)}^2\leq c(\epsilon)\cdot \lambda\cdot \ell(\alpha) 
    \end{equation*}
    where $c(\epsilon)>0$ depends only on $\epsilon$. 
\end{lemma}
\begin{proof}
    The proof is essentially the same as the proof of \cite[Proposition 12]{WX2022b}. Let $\alpha:[0,\ell] \to X_g$, $\ell=\ell(\alpha)$ be the arc-length parameterization such that $\alpha(0)=p$ and $\alpha(\ell)=q$. 
    For any $x\in\alpha$, by an argument similar to Lemma \ref{lem:inj} we may assume $\inj(x) > \epsilon/2$. Thus, by standard Sobolev embedding theorem \cite{Taylor2023}, 
	\begin{align*}
	    \norm{\nabla f}_{L^\infty(B(x,\epsilon/4))} 
        \leq \tilde{c}(\epsilon)\sum_{j=0}^{N} \norm{\Delta^j(\dif f)}_{L^2(B(x,\epsilon/2))}\leq \frac43 \tilde{c}(\epsilon)\norm{\dif f}_{L^2(B(x,\epsilon/2))}, 
	\end{align*}
	where $\tilde{c}(\epsilon)>0$ only depends on $\epsilon$. 

    If $\ell< \epsilon/4$, then since $\ell<1$, 
    \begin{equation}\label{eqn:oscillation-1}
        \begin{aligned}
            \abs{f(p)-f(q)}^2
            &\leq \norm{\nabla f}_{L^\infty(B(p,\epsilon/4))}^2\cdot \ell^2\leq \frac{16}{9}\tilde{c}(\epsilon)^2\cdot \lambda\cdot \ell. 
        \end{aligned}
    \end{equation}
    If $\ell\geq \epsilon/4$, then for each integer $n \in \{0,\cdots, \lfloor 4\ell/\epsilon \rfloor\}$, set $\alpha_n = \alpha(n\epsilon/4)$. We have 
    \begin{equation}\label{eqn:oscillation-2}
    \begin{aligned}
        \abs{f(p)-f(q)}
        &\leq \sum_{n=0}^{\lfloor 4\ell/\epsilon\rfloor-1}\abs{f(\alpha_n) - f(\alpha_{n+1})} + \abs{f(\alpha_{\lfloor 4\ell/\epsilon \rfloor}) - f(q)}\\
        &\leq \sum_{n=0}^{\lfloor 4\ell/\epsilon\rfloor} \norm{\nabla f}_{L^\infty(B(\alpha_n,\epsilon/4))}\cdot \frac{\epsilon}{4}\\
        &\leq \frac{1}{3}\tilde{c}(\epsilon)\sum_{n=0}^{\lfloor 4\ell/\epsilon\rfloor} \norm{\dif f}_{L^2(B(\alpha_n,\epsilon/2))}\cdot \epsilon\\
        &\leq \frac{1}{3}\tilde{c}(\epsilon) \bigg( \sum_{n=0}^{\lfloor 4\ell/\epsilon\rfloor} \norm{\dif f}^2_{L^2(B(\alpha_n,\epsilon/2))} \bigg)^{1/2} \bigg( \sum_{n=0}^{\lfloor 4\ell/\epsilon\rfloor} \epsilon^2 \bigg)^{1/2}. 
    \end{aligned}
    \end{equation}
    We claim that if $B(\alpha_{n_1},\epsilon/2)\cap B(\alpha_{n_2},\epsilon/2)\neq \emptyset$, then $\abs{n_1-n_2} < 2$. Otherwise, we can find a shorter geodesic segment joining $p$ and $q$ that does not intersect with any $\gamma\in \sN_{\epsilon}(X_g)$, which leads to a contradiction. It then follows that 
    \begin{equation}\label{eqn:oscillation-3}
        \sum_{n=0}^{\lfloor 4\ell/\epsilon\rfloor} \norm{\dif f}^2_{L^2(B(\alpha_n,\epsilon))} \leq 2\norm{\dif f}_{L^2(X_g)}^2 = 2\lambda. 
    \end{equation}
    Moreover, since $\epsilon<1$, 
    \begin{equation}\label{eqn:oscillation-4}
        \sum_{n=0}^{\lfloor 4\ell/\epsilon\rfloor} \epsilon^2 \leq \sum_{n=0}^{\lfloor 4\ell/\epsilon\rfloor} \epsilon = (\lfloor 4\ell/\epsilon\rfloor+1)\epsilon \leq 8\ell, 
    \end{equation}
    we deduce from \eqref{eqn:oscillation-2}, \eqref{eqn:oscillation-3}, and \eqref{eqn:oscillation-4} that 
    \begin{equation}\label{eqn:oscillation-5}
        \abs{f(p)-f(q)}^2 \leq \frac{16}{9}\tilde{c}(\epsilon)^2\cdot \lambda\cdot \ell. 
    \end{equation}
    The lemma then follows from \eqref{eqn:oscillation-1} and \eqref{eqn:oscillation-5} by setting $c(\epsilon) = \frac{16}{9}\tilde{c}(\epsilon)^2$. 
\end{proof}

Similarly to Lemma \ref{lem:mass}, we have 
\begin{lemma}\label{l-s-ub-3}
    Let $X_g$ be a closed hyperbolic surface, and let $f$ be an eigenfunction of $\Delta$ with eigenvalue $\lambda\leq1/4-\delta^2$. Then 
    \[ \int_{X_g} f^2 \dif A \leq 3\int_{\thick_g} f^2 \dif A, \] 
    where $\thick_g$ is the thick part of $X_g$ defined by \eqref{eqn:delta-thickthin}. 
\end{lemma}
\begin{proof}
    The proof is essentially the same as that in Lemma \ref{lem:mass}. Since 
    \[ \frac{1}{4}-\lambda\geq \delta^2, \]
    for any $\gamma\in\sN_{\epsilon}(X_g)$, the proof of \eqref{i-c-tt} gives that
    \[ \frac{\int_{\collar(\gamma,w(\gamma)-1/\delta)} f^2 \dif A}{\int_{\collar(\gamma,w(\gamma))} f^2 \dif A} \leq \frac{\int_{0}^{w(\gamma)-1/\delta} \cosh^2\!(\delta\rho) \dif\rho}{\int_{0}^{w(\gamma)} \cosh^2\!(\delta\rho) \dif\rho} \leq \frac{4}{e^2}, \] 
    thus 
    \begin{equation*}
        \begin{aligned}
            \int_{X_g} f^2 \dif A 
            & = \int_{\bigcup_{\gamma\in \sN_{\epsilon}(X_g)}\collar(\gamma,w(\gamma)-1/\delta)} f^2 \dif A + \int_{\thick_g} f^2 \dif A\\
            &\leq 2 \int_{\bigcup_{\gamma\in \sN_{\epsilon}(X_g)}\collar(\gamma,w(\gamma))\setminus\collar(\gamma,w(\gamma)-1/\delta)} f^2 \dif A + \int_{\thick_g} f^2 \dif A\\
            &\leq 3 \int_{\thick_g} f^2 \dif A. 
        \end{aligned}
    \end{equation*}
    The proof is complete. 
\end{proof}

Now let $\lambda\leq1/4-\delta^2$ be an eigenvalue of $\Delta$ with multiplicity $m(\lambda)$, and let $\{\phi_i\}_{1\leq i\leq m(\lambda)}$ be $m(\lambda)$ orthonormal eigenfunctions with eigenvalue $\lambda$. For any $x\in X_g$, we define a function $f_x$ by 
\[ f_x(z) \df \frac{\sum_{1\leq i\leq m(\lambda)}\phi_i(x)\phi_i(z)}{\sqrt{\sum_{1\leq i\leq m(\lambda)}\phi_i(x)^2}}. \] 
Then clearly $f_x(z)$ is a normalized eigenfunction of $\Delta$ with eigenvalue $\lambda$ satisfying 
\[ f_x(x)^2 = \sum_{i=1}^{m(\lambda)}\phi_i(x)^2. \] 

We consider the thick-thin decomposition of $X_g$ defined by \eqref{eqn:delta-thickthin}.
Inspired by the recent work \cite{GLN2024} of Gross--Lachman--Nachmias, for any $j\geq0$, we also set 
\[ \mu_{j} = 2^{j}\sqrt{\lambda} \text{\quad and\quad} R_j = c_{1}\cdot 2^{j}/\sqrt{\lambda}. \] 
The constant $c_{1}>0$, depending only on $\epsilon$, will be determined in Lemma \ref{lem:pf1}. Now fix some $k\in[1,I_\epsilon]$. We first give two definitions. 
\begin{definition}\label{def:G_j}
    For each $j\geq0$, define 
    \[ G_j \df \Big\{ x\in \sA_k: \ \mu_{j-1}\leq f_x(x)^2 < \mu_{j} \Big\}, \] 
    where we set $\mu_{-1}=0$.  
\end{definition}
\begin{definition}\label{def:B_x}
    For any $x\in \sA_k$ and any $R>0$, define 
    \begin{equation*}
        B_{x,R} \df 
        \left\{z\in \sA_k\cup\sS_k:\ \begin{gathered}
            \text{there exists a geodesic segment of length $\leq R$ joining}\\ 
            \text{$x$ and $z$ that does not intersect with any $\gamma\in \sN_{\epsilon}(X_g)$} 
        \end{gathered}
        \right\}. 
    \end{equation*}
\end{definition}

We have the following lemmas. 
\begin{lemma}\label{lem:pf1}
    Suppose $x\in G_j$. If $j\geq1$, then 
    \begin{equation*}
        \area(B_{x,R_j})\leq \frac{4}{\mu_{j-1}}. 
    \end{equation*}
\end{lemma}
\begin{proof}
    We first prove that 
    \begin{equation}\label{eqn:pf1}
        B_{x,R_j} \subset \Big\{z\in X_g:\ f_x(z)\geq \frac{f_x(x)}{2} \Big\}.
    \end{equation}
    Otherwise, suppose that there would exist some $z_0\in B_{x,R_j}$ such that 
    \[ f_x(z_0) < \frac{f_x(x)}{2}. \] 
    By the definition of $B_{x,R_j}$, we can find a shortest geodesic segment of length $\leq R_j$ joining $x$ and $z_0$ that does not intersect with any $\gamma\in \sN_{\epsilon}(X_g)$. Thus, by Lemma \ref{lem:oscillation} we have 
    \[ \frac{\mu_{j-1}}{4}\leq \frac{f_x(x)^2}{4}\leq \abs{f_x(x) - f_x(z_0)}^2 \leq c_{2}\cdot R_j\cdot \lambda, \] 
    where $c_2=c(\epsilon)$ depends only on $\epsilon$. It follows that 
    \[ 1\leq 8c_{2}\cdot c_{1}. \] 
    If we set $c_{1} = 1/(9c_{2})$, then it leads to a contradiction. Thus, it follows from  \eqref{eqn:pf1} and the Markov inequality that for any $x\in G_j$,
    \begin{equation*}
    \area(B_{x,R_j})\leq  \area\left(\Big\{z\in X_g:\ f_x(z)\geq \frac{f_x(x)}{2} \Big\} \right)\leq \frac{\int_{X_g}f_x^2(z)}{\left(\frac{f_x(x)}{2}\right)^2}\leq \frac{4}{\mu_{j-1}}.
    \end{equation*}
    This completes the proof.
\end{proof}

We will use a different notion of diameter which is defined as follows.  
\begin{definition}\label{def:diam}
    \begin{equation*}
        \diam_c(\sA_k) \df \sup_{x,z\in \sA_k}
        \left\{\length(\alpha):\ \begin{gathered}
            \text{$\alpha$ is a shortest geodesic segment joining $x$ and $z$}\\ 
            \text{that does not intersect with any $\gamma\in \sN_{\epsilon}(X_g)$} 
        \end{gathered}
        \right\}.
    \end{equation*}
\end{definition}
\begin{lemma}\label{lem:pf2}
    Suppose $x\in \sA_k$ and $R>2c_1$. If $R < \diam_c(\sA_k)$, then 
    \[ \area(B_{x,R/2}) \geq c_{3}\cdot R, \] 
    where $c_3>0$ depends only on $\epsilon$. 
\end{lemma}
\begin{proof}
    Since $R<\diam_c(\sA_k)$, by definition there exist some $z\in \sA_k$ and a shortest geodesic segment $\alpha$ from $x$ to $z$ of length $R/2$ that does not intersect with any $\gamma\in \sN_{\epsilon}(X_g)$. 
    By Lemma \ref{lem:inj}, we know that $\alpha\subset\sA_k\cup\sS_k$, hence $$\alpha\subset B_{x,R/2}.$$ 
    Let $\alpha:[0,R/2] \to X_g$ be a geodesic of arc-length parameter with $\alpha(0)=x$. Then for each integer $n \in \{0,\cdots, \lfloor R/c_1 \rfloor - 1\}$, we have 
	\[ B(\alpha(nc_1/2),c_1/4)\subset B_{x,R/2}, \]
	and the geodesic balls $B(\alpha(nc_1/2),c_1/4)$, $1\leq n\leq \lfloor R/c_1\rfloor-1$, are pairwise disjoint. Moreover, by Lemma \ref{lem:inj} we have that for all $0\leq n\leq \lfloor R/c_1\rfloor-1$,
    \[\inj(\alpha(nc_1/2))>\epsilon.\]
    Therefore, 
	\begin{equation*}
	     \begin{aligned}
        	\area(B_{x,R/2})
            &\ge \sum_{n=0}^{\lfloor R/c_1\rfloor - 1}\area\big(B(\alpha(nc_1/2),c_1/4)\big) \\
            &\ge \min\Big\{\pi\epsilon^2, \frac{\pi c_1^2}{16} \Big\}\cdot \lfloor R/c_1\rfloor\\
            &\ge c_3\cdot R, 
        \end{aligned}
	\end{equation*}
   where $c_3$ depends only on $\epsilon$. The proof is complete.
\end{proof}

Now we give an upper bound of the integral of $f_x(x)^2$ over $\sA_k$. 
\begin{lemma}\label{lem:pf3}
    There exists some constant $C>0$ depending only on $\epsilon$ such that 
    \begin{equation*}
        \int_{\sA_k} f_x(x)^2 \leq \max\Big\{ C\sqrt{\lambda}\cdot \area(\sA_k\cup\sS_k),\ 8 \Big\}. 
    \end{equation*}
\end{lemma}
\begin{proof}
    Let $j_0\geq0$ be the maximal index such that $G_{j_0}\neq\emptyset$ (see Definition \ref{def:G_j}). If $j_0=0$, then 
    \begin{equation}\label{eqn:bound0}
        \int_{\sA_k} f_x(x)^2 \leq \mu_0\cdot\area(G_0) = \sqrt{\lambda}\cdot\area(\sA_k). 
    \end{equation}
    From now on, we always assume that $j_0\geq1$. Note that 
    \begin{equation}\label{ub-g1}
        \int_{\sA_k} f_x(x)^2 \leq \sum_{j=0}^{j_0}\mu_j\cdot\area(G_j).
    \end{equation}

    \underline{Case-1}. $R_{j_0}\geq \diam_c(\sA_k)$. Then by definition we know that for any $x\in G_{j_0}$
    \begin{equation*}
        \sA_k \subset B_{x,R_{j_0}}. 
    \end{equation*}
    Since $j_0\geq1$, it follows from Lemma \ref{lem:pf1} that 
    \begin{equation*}
        \area(\sA_k) \leq \area(B_{x,R_{j_0}}) \leq \frac{4}{\mu_{j_0-1}},
    \end{equation*}
    we thus have 
    \begin{equation*}
        \mu_{0}< \mu_{1}< \cdots< \mu_{j_0}\leq \frac{8}{\area(\sA_k)}. 
    \end{equation*}
    By the definition of $G_j$ we thus have 
    \begin{equation}\label{eqn:bound1}
        \sum_{j=0}^{j_0}\mu_j\cdot\area(G_j) \leq \frac{8}{\area(\sA_k)}\sum_{j=0}^{j_0}\area(G_j) = 8. 
    \end{equation}
    
    \underline{Case-2}. $R_{j_0}< \diam_c(\sA_k)$. For any $1\leq j\leq j_0$, let $\{B_{x_i,R_j/2}\}_{1\leq i\leq l}$ be a maximal collection of disjoint sets such that $x_i\in G_j$, $\Forall 1\leq i\leq l$. By Lemma \ref{lem:pf2} we have
    \begin{equation}\label{ub-f-l}
    l\leq \frac{\area(\sA_k \cup \sS_k)}{c_3R_j}.
    \end{equation}
    Next we claim that 
    \begin{equation}\label{eqn:pf3-1}
        G_j\subset \bigcup_{1\leq i\leq l} B_{x_i,R_j}. 
    \end{equation}
    Otherwise, there would exist some $x\in G_j$ such that any geodesic segment joining $x$ and $x_i$ that does not intersect with any $\gamma\in \sN_{\epsilon}(X_g)$ has length $\geq R_j$. Thus, $B_{x,R_j/2}$ is disjoint with any $B_{x_i,R_j/2}$, leading to a contradiction because our choice of $\{B_{x_i,R_j/2}\}_{1\leq i\leq l}$ is maximal. 
    Combining \eqref{eqn:pf3-1}, Lemma \ref{lem:pf1}, Lemma \ref{lem:pf2} and \eqref{ub-f-l}, we deduce that 
    \begin{equation*}
        \area(G_j)
        \leq \sum_{i=1}^{l} \area(B_{x_i,R_j})
        \leq \sum_{i=1}^{l} \frac{4}{\mu_{j-1}}
        \leq \frac{\area(\sA_k\cup\sS_k)}{c_3R_j}\frac{4}{\mu_{j-1}}.
    \end{equation*}
    Therefore, we have 
    \begin{equation}\label{eqn:bound2}
        \begin{aligned}
            \sum_{j=0}^{j_0}\mu_j\cdot\area(G_j) 
            &\leq \mu_0\cdot\area(G_0) + \sum_{j=1}^{j_0}\mu_j\cdot\frac{4\area(\sA_k\cup\sS_k)}{c_3\mu_{j-1}R_{j}}\\
            &\leq \sqrt{\lambda}\cdot\area(\sA_k) + \sum_{j=1}^{j_0}\frac{8\sqrt{\lambda}\area(\sA_k\cup\sS_k)}{c_1c_32^{j}}\\
            &\leq C\sqrt{\lambda}\cdot \area(\sA_k\cup\sS_k). 
        \end{aligned}
    \end{equation}  
    The lemma then follows from \eqref{eqn:bound0}, \eqref{ub-g1}, \eqref{eqn:bound1} and \eqref{eqn:bound2}. 
\end{proof}

Now we are ready to prove Theorem \ref{thm:mt-2}. 
\begin{proof}[Proof of Theorem \ref{thm:mt-2}]
    By Lemma \ref{l-s-ub-3} and Lemma \ref{lem:pf3} we have 
    \begin{equation*}
        \begin{aligned}
            m(\lambda) = \int_{X_g} \sum_{i=1}^{m(\lambda)} \phi_i(x)^2
            &\leq 3 \int_{\thick_g} \sum_{i=1}^{m(\lambda)} \phi_i(x)^2\\
            &  =  3 \sum_{k=1}^{I_\epsilon} \int_{\sA_k} f_x(x)^2\\
            &\leq 3 \sum_{k=1}^{I_\epsilon} \max\Big\{ C\sqrt{\lambda}\cdot \area(\sA_k\cup\sS_k),\ 8 \Big\}\\
            &\leq 3C\sqrt{\lambda}\cdot \area(X_g) + 24I_\epsilon. 
        \end{aligned}
    \end{equation*}
    The proof is then finished by Gauss--Bonnet.  
\end{proof}

\bibliographystyle{amsalpha}
\bibliography{ref}

@article{Besson1980,
  title = {Sur la multiplicit{\'e} de la premi{\`e}re valeur propre des surfaces {R}iemanniennes},
  author = {G{\'e}rard Besson},
  journal = {Annales de l'Institut Fourier},
  year = {1980},
  volume = {30},
  number = {1},
  pages = {109--128},
  doi = {10.5802/aif.777}
}

@book{Bergeron2016,
  title = "{The Spectrum of Hyperbolic Surfaces}",
  author = {Nicolas Bergeron},
  series = {Universitext},
  publisher = {Springer, Cham; EDP Sciences},
  address = {Les Ulis},
  year = {2016},
  pages = {xiii+370},
  note = {Appendix C by Valentin Blomer and Farrell Brumley, Translated from 2011 French language edition by Farrell Brumley},
  mrnumber = {2857626}
}

@article{BP2023,
  title = {Linear programming bounds for hyperbolic surfaces}, 
  author = {Fortier Bourque, Maxime and Bram Petri},
  journal = {preprint},
  year = {2023},
  pages = {arxiv:2302.02540},
  eprint = {2302.02540}
}

@article{BGMPP23,
  title={Two counterexamples to a conjecture of {C}olin de {V}erdi\`ere on multiplicity},
  author={Fortier Bourque, Maxime and Gruda-Mediavilla, {\'E}mile and Bram Petri and Mathieu Pineault},
  journal={arXiv preprint},
  year={2023},
  pages={arXiv:2312.03504}
}

@book{Buser1992,
  title = "{Geometry and Spectra of Compact Riemann Surfaces}",
  author = {Peter Buser},
  series = {Progress in Mathematics},
  volume = {106},
  publisher = {Birkh\"auser Boston, Inc.},
  address = {Boston, MA},
  year = {1992},
  pages = {xiv+454},
  isbn = {0-8176-3406-1},
  mrnumber = {1183224}
}

@article{Cheng1975,
  title = {Eigenvalue comparison theorems and its geometric applications},
  author = {Cheng, Shiu-yuen},
  journal = {Mathematische Zeitschrift},
  year = {1975},
  volume = {143},
  pages = {289-297}
}

@article{Cheng1976,
  title = {Eigenfunctions and nodal sets},
  author = {Cheng, Shiu-yuen},
  journal = {Commentarii Mathematici Helvetici},
  year = {1976},
  volume = {51},
  number = {},
  pages = {43--55},
  doi = {10.1007/BF02568142}
}

@article{CCdV88,
  title={Sur la multiplicit{\'e} de la premi{\`e}re valeur propre d’une surface de {R}iemann {\`a} courbure constante},
  author={Bruno Colbois and Colin de Verdi{\`e}re, Yves},
  journal={Commentarii Mathematici Helvetici},
  volume={63},
  pages={194--208},
  year={1988},
  publisher={Springer}
}

@article{CdV86,
  title = {Sur la multiplicit{\'e} de la premiere valeur propre non nulle du {L}aplacien},
  author = {Colin de Verdi{\`e}re, Yves},
  journal = {Commentarii Mathematici Helvetici},
  year = {1986},
  volume = {61},
  pages = {254--270}
}

@article{Gamburd2002,
  title = {On the spectral gap for infinite index “congruence” subgroups of {$\mathrm{SL}_2(\mathbb{Z})$}},
  author = {Alex Gamburd},
  journal = {Israel Journal of Mathematics},
  year = {2002},
  volume = {127},
  pages = {157-200}
}

@article{GLeMST21,
  title = {Short geodesic loops and {$L^p$} norms of eigenfunctions on large genus random surfaces},
  author = {Clifford Gilmore and Le Masson, Etienne and Tuomas Sahlsten and Joe Thomas},
  journal = {Geometric and Functional Analysis},
  year = {2021},
  volume = {31},
  number = {},
  pages = {62-110},
  doi = {10.1007/s00039-021-00556-6},
}

@article{GLN2024,
  title = {A sharp lower bound on the small eigenvalues of surfaces},
  author = {Renan Gross and Guy Lachman and Asaf Nachmias},
  journal = {preprint},
  year = {2024},
  pages = {arxiv:2407.21780},
  eprint = {2407.21780},
  archivePrefix = {arXiv}
}

@article {HW25,
  AUTHOR = {He, Yuxin and Wu, Yunhui},
  TITLE = {On second eigenvalues of closed hyperbolic surfaces for large genus},
  JOURNAL = {Journal of Differential Geometry},
  VOLUME = {130},
  YEAR = {2025},
  NUMBER = {2},
  PAGES = {403--441}
}

@article{JTYZZ2021,
  title = {Equiangular lines with a fixed angle},
  author = {Zilin Jiang and Jonathan Tidor and Yuan Yao and Shengtong Zhang and Yufei Zhao},
  journal = {Annals of Mathematics},
  year = {2021},
  volume = {194},
  number = {3},
  pages = {729-743},
  doi = {10.4007/annals.2021.194.3.3}
}

@article {LM2024,
  title = {Maximal multiplicity of {L}aplacian eigenvalues in negatively curved surfaces},
  author = {Cyril Letrouit and Simon Machado},
  journal = {Geometric and Functional Analysis},
  year = {2024},
  volume = {34},
  number = {5},
  pages = {1609--1645}
}

@article{Mondal2015,
  title = {On Topological Upper-Bounds on the Number of Small Cuspidal Eigenvalues of Hyperbolic Surfaces of Finite Area},
  author = {Sugata Mondal},
  journal = {International Mathematics Research Notices},
  year = {2015},
  volume = {2015},
  number = {24},
  pages = {13208-13237},
  doi = {10.1093/imrn/rnv092}
}

@article{Monk22,
  title={Benjamini--Schramm convergence and spectra of random hyperbolic surfaces of high genus},
  author={Monk, Laura},
  journal={Analysis \& PDE},
  volume={15},
  number={3},
  pages={727--752},
  year={2022},
  publisher={Mathematical Sciences Publishers}
}

@article{Nadirashvili1988,
  title={Multiple eigenvalues of the {L}aplace operator},
  author={Nadirashvili, Nikolai Semenovich},
  journal={Mathematics of the USSR-Sbornik},
  volume={61},
  number={1},
  pages={225},
  year={1988},
  publisher={IOP Publishing}
}

@article{RY68,
  title={Solution of the {H}eawood map-coloring problem},
  author={Ringel, Gerhard and Youngs, John W.T.},
  journal={Proceedings of the National Academy of Sciences},
  volume={60},
  number={2},
  pages={438--445},
  year={1968}
}

@article{Sevennec2002,
  title = {Multiplicity of the second {S}chr{\"o}dinger eigenvalue on closed surfaces},
  author = {Bruno S{\'e}vennec},
  journal = {Mathematische Annalen},
  year = {2002},
  volume = {324},
  number = {},
  pages = {195-211},
  doi = {10.1007/s00208-002-0337-1}
}

@inproceedings{SWY80,
  title = {Geometric bounds on the low eigenvalues of a compact surface},
  author = {Richard Schoen and Scott Wolpert and Shing-Tung Yau},
  booktitle = {Geometry of the {L}aplace operator},
  series = {Proceedings of  Symposia in Pure Mathematics},
  volume = {XXXVI},
  publisher = {American Mathematical Society},
  address = {Providence, R.I.},
  year = {1980},
  pages = {279--285},
  mrnumber = {573440}
}

@book{Taylor2023,
  title = "{Partial Differential Equations I}",
  subtitle = "{Basic Theory}",
  author = {Michael E. Taylor},
  series = {Applied Mathematical Sciences},
  volume = {115},
  edition = {3},
  publisher = {Springer Cham},
  address = {},
  year = {2023},
  pages = {xxiv+714}
}

@article{WX2022b,
  title = {Optimal lower bounds for first eigenvalues of {R}iemann surfaces for large genus},
  author = {Yunhui Wu and Yuhao Xue},
  journal = {American Journal of Mathematics},
  year = {2022},
  volume = {144},
  number = {4},
  pages = {1087--1114}
}

@article{WZ2025,
  title = {Spectral gaps on thick part of moduli spaces},
  author = {Yunhui Wu and Haohao Zhang},
  journal = {preprint},
  year = {2025},
  pages = {arxiv:2501.09266},
  eprint = {2501.09266}
}

\end{document}